\documentclass[a4paper,UKenglish,numberwithinsect,autoref,final]{lipics-v2019}

\nolinenumbers
\hideLIPIcs

\usepackage[usenames,dvipsnames,table]{xcolor}
\usepackage{xcolor}

\usepackage{amsmath, amsthm, amssymb, mathrsfs, calc}
\usepackage{array}
\usepackage{graphicx}
\usepackage{multirow}
\usepackage{marvosym}
\usepackage{mathtools}
\usepackage{verbatim}
\usepackage{url}
\usepackage{setspace}
\usepackage{mleftright}
\usepackage{framed}

\usepackage{microtype}
\usepackage[small,basic]{complexity}
\usepackage{xspace}
\usepackage{bm}
\usepackage{aliascnt}
\usepackage{hyperref}
\hypersetup{final}

\usepackage{ifdraft}
\usepackage[colorinlistoftodos,prependcaption,textsize=tiny,obeyDraft]{todonotes}
\newcommand{\gk}[1]{\todo[color=green!40!]{#1}}

\colorlet{shadecolor}{blue!20}

\newaliascnt{conj}{theorem}
 \newtheorem{conj}[conj]{Conjecture}
 \aliascntresetthe{conj}

\theoremstyle{definition}
\newaliascnt{obj}{theorem}
 \newtheorem{obj}[obj]{Objective}
 \aliascntresetthe{obj}

\newaliascnt{problem}{theorem}
 \newtheorem{problem}[problem]{Problem}
 \aliascntresetthe{problem}

\numberwithin{equation}{section}

\newcommand{\seq}[3][{}]{\langle #2 \rangle_{#3}^{#1}}

\renewcommand{\Re}{\operatorname{Re}}

\renewcommand{\S}{\textsection}

\renewcommand{\R} {\mathbb{R}} \newcommand{\Z} {\mathbb{Z}} \renewcommand{\C} {\mathbb{C}} \newcommand{\N} {\mathbb{N}}    \newcommand{\Q} {\mathbb{Q}}

\DeclareMathOperator{\sign}{sign}

\providecommand*{\eu}%
{\ensuremath{\mathrm{e}}}
\providecommand*{\iu}%
{\ensuremath{\mathrm{i}}}

\makeatletter
\DeclareRobustCommand\bigop[1]{%
  \mathop{\vphantom{\sum}\mathpalette\bigop@{#1}}\slimits@
}
\newcommand{\bigop@}[2]{%
  \vcenter{%
    \sbox\z@{$#1\sum$}%
    \hbox{\resizebox{\ifx#1\displaystyle.9\fi\dimexpr\ht\z@+\dp\z@}{!}{$\m@th#2$}}%
  }%
}
\makeatother

\newcommand{\KF}{\DOTSB\bigop{\mathrm{K}}}

\newcommand{\hypgeo}[2]{%
  \operatorname{%
    {\vphantom{\mathnormal{F}}}_{#1}%
    \kern-\scriptspace
    \mathnormal{F}_{#2}%
  }%
}

\renewcommand{\P}{\ensuremath\mathbb{P}}

\renewcommand{\epsilon}{\ensuremath\varepsilon}
\newclass{\PTIME}{PTIME}
\newclass{\PCF}{PCF}

\newcommand{\nats}{\mathbb{N}}

\AtBeginDocument{%
  \DeclareFontShape{T1}{lmr}{m}{scit}{<->ssub*lmr/m/scsl}{}%
}
\usepackage[T1]{fontenc}

\title{On Positivity and Minimality for Second-Order Holonomic Sequences}
\titlerunning{On Positivity and Minimality for Second-Order Holonomic Sequences}

\author{George Kenison}%
{Department of Computer Science, University of Oxford, Oxford, United Kingdom}%
{george.kenison@cs.ox.ac.uk}{}{}

\author{Oleksiy Klurman}%
{Max Planck Institute for Mathematics, Bonn, Germany}{lklurman@gmail.com}{}{}

\author{Engel Lefaucheux}%
{Max Planck Institute for Software Systems, Saarland Informatics Campus, Germany}%
{elefauch@mpi-sws.org}{}{}

\author{Florian Luca}%
{School of Mathematics, University of the Witwatersrand, South Africa
\and
Max Planck Institute for Mathematics, Bonn, Germany}%
{florian.luca@wits.ac.za}{}{}
	
\author{Pieter Moree}%
{Max Planck Institute for Mathematics, Bonn, Germany}%
{moree@mpim-bonn.mpg.de}%
{0000-0002-5318-2587}{}

\author{Jo\"el Ouaknine}%
{Max Planck Institute for Software Systems, Saarland Informatics Campus, Germany
\and
Department of Computer Science, University of Oxford, Oxford, United Kingdom}%
{joel@mpi-sws.org}{}{}

\author{Markus A. Whiteland}%
{Max Planck Institute for Software Systems, Saarland Informatics Campus, Germany}%
{mawhit@mpi-sws.org}%
{0000-0002-6006-9902}{}

\author{James Worrell}%
{Department of Computer Science, University of Oxford, Oxford, United Kingdom}%
{jbw@cs.ox.ac.uk}{}{}

\authorrunning{G. Kenison et al.}

\Copyright{George Kenison, Oleksiy Klurman, Engel Lefaucheux, Florian Luca, Pieter Moree, Jo\"el Ouaknine, Markus A. Whiteland, and James Worrell}



\begin{document}

\maketitle

\begin{abstract}
  An infinite sequence $\seq{u_n}{n\in\N}$ of real numbers is \emph{holonomic} (also
  known as \emph{$P$-recursive} or \emph{$P$-finite}) if it satisfies a linear
  recurrence relation with polynomial coefficients. Such a sequence is
  said to be \emph{positive} if each $u_n \geq 0$, and \emph{minimal} if, given any other
  linearly independent sequence $\seq{v_n}{n\in\N}$ satisfying the same
  recurrence relation, the ratio $u_n/v_n$ converges to $0$.

  In this paper, we focus on holonomic sequences satisfying a
  second-order recurrence
\[ g_3(n)u_n = g_2(n)u_{n-1} + g_1(n)u_{n-2} \, , \] 
where each coefficient
  $g_3, g_2,g_1 \in \mathbb{Q}[n]$ is a polynomial of degree at
  most $1$. We establish two main results. First, we show that
  deciding positivity for such sequences reduces to deciding
  minimality. And second, we prove that deciding minimality is
  equivalent to determining whether certain numerical expressions
  (known as \emph{periods}, \emph{exponential periods}, and
  \emph{period-like integrals}) are equal to zero. Periods and related
  expressions are classical objects of study in algebraic geometry and number
  theory, and several established conjectures (notably those of
  Kontsevich and Zagier) imply that they have a decidable equality
  problem, which in turn would entail decidability of Positivity and
  Minimality for a large class of second-order holonomic sequences.
\end{abstract}

\section{Introduction} \label{sec: introduction}

\gk{Draft mode on.  Use global option \emph{final} to hide todo notes and labels. Toggle line numbers with \texttt{\textbackslash{nolinenumbers}}}

\emph{Holonomic} sequences (also known as \emph{\(P\)-recursive} or
  \emph{\(P\)-finite} sequences) are infinite sequences of real (or complex) numbers that satisfy a
  linear recurrence relation with polynomial coefficients. The
  earliest and best-known example is the Fibonacci sequence, given by
  Leonardo of Pisa in the 12th century;
more recently, Ap\'ery famously
  made use of certain holonomic sequences satisfying the recurrence relation 
\begin{equation*} \label{eq:apery}
  (n+1)^3 u_{n+1} = (34n^3 + 51n^2 + 27n + 5)u_n - n^3 u_{n-1} \quad (n\in\N)
\end{equation*}
to prove that \(\zeta(3) := \sum_{n=1}^\infty n^{-3}\) is
irrational~\cite{apery1979irrationalite}. Holonomic sequences now form a vast subject in their own
right, with numerous applications in mathematics and
other sciences; see, for instance, the monographs \cite{aeqb,everest2003recurrence,flajolet2009analytic}
or the seminal paper~\cite{Zei90}.

Formally, a holonomic recurrence is a relation of the following form:
\[
g_{k+1}(n)u_{n+k} = g_k(n)u_{n+k-1} + \ldots + g_1(n)u_n \, ,
\]
where $g_{k+1}, \ldots, g_1 \in \mathbb{Q}[n]$ are polynomials with
rational coefficients. We define the \emph{order} of the recurrence to
be $k$, and its \emph{degree} to be the maximum degree of the
polynomials $g_i$. Assuming that $g_{k+1}(n) \neq 0$ for non-negative
integer $n$, the above recurrence uniquely defines an infinite
sequence $\seq[\infty]{u_n}{n=0}$ once the $k$ initial values
$u_0, \ldots, u_{k-1}$ are specified.\footnote{In the sequel, it will
  in fact often be convenient to start the sequence at $u_{-1}$
  instead of $u_0$.} Such a sequence is said to be \emph{holonomic},
and---in slight abuse of terminology---will be understood to inherit
the \emph{order} and \emph{degree} of its defining
recurrence. Degree-$0$ holonomic sequences---i.e., such that all
polynomial coefficients appearing in the recurrence relation are
constant---are also known as \emph{\(C\)-finite}
sequences, and first-order holonomic sequences are known as
\emph{hypergeometric} sequences.

Holonomic sequences naturally give rise to \emph{holonomic functions} by
considering the associated generating power series $\mathcal{F}(x) =
\sum_{n=0}^{\infty} u_n x^n$. As is well-known, the generating
functions of $C$-finite sequences are rational functions, and those of
hypergeometric sequences are hypergeometric functions. Properties of
holonomic functions---and in particular the differential equations
that they obey---will play a central r\^ole in our analysis of their
defining sequences.

There is a voluminous literature devoted to the study of identities
 for holonomic sequences. However, as noted by Kauers and Pillwein,
\emph{``in contrast, [\ldots] almost no algorithms are available for inequalities''}~\cite{KP10}.
For example, the \emph{Positivity Problem} (i.e., whether every
term of a given sequence is non-negative) for \(C\)-finite
sequences
is only known to be decidable at
low orders, and there is strong evidence that the problem is
mathematically intractable in general~\cite{ouaknine2014positivity, ouaknine2015linear}; 
see
also~\cite{halava2006positivity,laohakosol2009positivity,ouaknine2014positivity,OW14a}. 
For holonomic sequences that are not \(C\)-finite, virtually no decision
procedures currently exist for Positivity, although several partial results
and heuristics are
known (see, for example
\cite{Liu2010positivity,KP10,mezzarobba2010effective,XY11,Pil13,PS15}).

Another extremely important property of holonomic sequences is \emph{minimality}; a sequence
$\seq{u_n}{n}$ is minimal if, given any other linearly independent sequence $\seq{v_n}{n}$ satisfying the
same recurrence relation, the ratio $u_n/v_n$ converges to $0$. Minimal holonomic sequences
play a crucial r\^{o}le, among others, in numerical calculations and asymptotics, as noted for
example in \cite{gautschi1967computational,Gau77,Gau81,DS07,AST07,DST10}---see also the references therein.
Unfortunately, there is also ample evidence
that determining algorithmically whether a given holonomic sequence is minimal is a very
challenging task, for which no satisfactory solution is at present
known to exist.

The systematic study of Positivity and Minimality for holonomic
sequences of order two and above is a vast undertaking.\footnote{At
  order one, both problems are algorithmically trivial: indeed, the
  positivity of a hypergeometric sequence is readily determined by
  inspecting the polynomial coefficients of its defining recurrence,
  together with the sign of the first few values of the sequence;
  and since the solution set of a hypergeometric recurrence is a
  one-dimensional vector space, such recurrences cannot possibly admit
  minimal sequences.}  Accordingly, our focus in the present paper is
on second-order, degree-1 sequences.\footnote{Positivity and
  minimality for second-order $C$-finite sequences can
  straightforwardly be determined from their closed-form solutions;
  see~\cite{halava2006positivity}.}  The generating functions of such
sequences satisfy certain linear differential equations, whose
solutions involve integrals of a particular shape; depending
on the original sequence, the definite forms of these integrals
are known either
as \emph{periods}, \emph{exponential periods}, or \emph{period-like
  integrals}. Periods and related expressions are classical objects of
study in algebraic geometry and number theory, and several established
conjectures---notably those of Kontsevich and
Zagier~\cite{kontsevich2001periods}---imply that they have a
decidable equality problem (see Appendix~\ref{app:Conjectures} for a
more detailed account of these facts and considerations). At a high level, whether a given holonomic
sequence is minimal or not is related to the radius of convergence of
its associated generating function, which in turn hinges on the
precise value of these definite integrals. Consequently, we reduce
the problem of determining minimality of a given sequence to whether
the corresponding integral is zero. Unfortunately, for holonomic
sequences of order greater than two, or of degree higher than one, solving
the attendant differential equations no longer yields integrals of the
appropriate shape.\\[1ex]
\noindent \textbf{\textsf{Main results.}}
We summarise our main results as follows. Consider the class of
real-algebraic, second-order, degree-1 holonomic sequences. For this class:
\begin{enumerate}

\item The Positivity Problem reduces to the Minimality Problem (Theorem~\ref{th:oracleRel}).

\item The Minimality Problem reduces to determining whether a period,
  an exponential period, or a period-like integral is equal to zero
  (Theorem~\ref{thm:MinimalityRedToPeriodLike}).
 
\end{enumerate}

\section{Preliminaries}

\subsection{Second-order linear recurrences}
\label{sub:linrec}

We study the 
behaviour of solutions
\(\seq[\infty]{v_n}{n=-1}\) to second-order recurrence relations of the form
	\begin{subequations} \label{eq: ogform}
		\begin{align}
		g_3(n)v_n &= g_2(n)v_{n-1} + g_1(n)v_{n-2}, \quad \text{or} \label{eq: ogforma} \\
		v_n &= \frac{g_2(n)}{g_3(n)} v_{n-1} + \frac{g_1(n)}{g_3(n)} v_{n-2},  \quad n\in\N. \label{eq: ogformb}
		\end{align}
	\end{subequations}
where \(g_1,g_2,g_3\in\mathbb{Q}[x]\).
Solutions to such recurrences are called \emph{holonomic sequences}.
In the sequel it is useful to transform recurrence \eqref{eq: ogform} as follows.
For \(g_3\in\mathbb{Q}[x]\), let \(\seq[\infty]{u_n}{n=-1}\) and \(\seq{v_n}{n}\) be real-valued sequences such that \(u_{-1}=v_{-1}\) and \(u_n = g_3(n) \cdots g_3(0) v_n\) for \(n\in\N_0\).  
Then it is easily seen that \(\seq{v_n}{n}\) is a solution to \eqref{eq: ogform} if and only if \(\seq{u_n}{n}\) is a solution to the recurrence 
%
%
	\begin{equation} \label{eq: sform}
		u_n = g_2(n) u_{n-1} + g_1(n)g_3(n-1) u_{n-2}. 
	\end{equation}
%
%
With this transformation we can translate statements about minimality and positivity of solutions to \eqref{eq: sform}, subject to the condition that \(g_3(n)>0\) for each \(n\in\N_0\).
If \(g_3(n)>0\) for each \(n\in\N_0\) then  \(\seq{u_n}{n}\) is minimal (positive) if and only if \(\seq{v_n}{n}\) is minimal (positive).

%
%
%

Let $\seq{u_n}n$ be a sequence satisfying the second-order
relation~\eqref{eq: ogform}. Note that if $g_2$ is identically $0$,
then $\seq{u_n}n$ consists of the interleaving of two hypergeometric
sequences, in which case positivity of $\seq{u_n}n$ is simply
equivalent to the positivity of both individual hypergeometric
sequences, something which can readily be determined as noted in the
Introduction. Moreover, such a recurrence admits no minimal
solutions: this follows straightforwardly from the observation that the
limit $\lim_{n\to \infty} A_n/B_n$ does not exist for the linearly
independent solutions $\seq[\infty]{A_n}{n=-1}$ and
$\seq[\infty]{B_n}{n=-1}$ defined by $A_{-1} = 1$, $A_0 = 0$,
$B_{-1} = 0$, $B_{0} = 1$. Similarly, if $g_1$ is identially $0$, then positivity and minimality
of $\seq{u_n}n$ likewise become trivial. In what follows, we will therefore
assume that none of $g_1,g_2,g_3$ are identically $0$.

Moreover (considering a shifted recurrence relation if
necessary), we can assume without loss of generality that each
polynomial coefficient has constant sign, and has no roots for $n\geq
0$.
Additionally we can assume that \(\sign(g_3)=1\).
The \emph{signature} of this relation is defined as the ordered tuple 
\((\sign(g_2(n)), \sign(g_1(n)))\).

\subsection{Asymptotic equalities for second-order linear recurrences}

Here we state asymptotic results established by Poincar\'{e} and Perron in the restricted setting of second-order recurrence relations.
Let \(\seq[\infty]{a_n}{n=1}\) and \(\seq[\infty]{b_n}{n=1}\) be real-valued sequences.
We say that \(u_n = b_n u_{n-1} + a_n u_{n-2}\) is a \emph{Poincar\'{e} recurrence} if the limits \(\lim_{n\to\infty} a_n = a\) and \(\lim_{n\to\infty} b_n= b\) exist and are finite.
The next result, initially considered by Poincar\'{e}
\cite{poincare1885difference} and expanded upon by Perron
\cite{perron1909poincare},  considers Poincar\'{e} recurrences as
perturbations of \(C\)-finite recurrences.

\begin{theorem}[Poincar\'{e}--Perron Theorem] \label{thm: poincareperron}
Suppose that \(u_n = b_n u_{n-1} + a_n u_{n-2}\) is a Poincar\'{e} recurrence and \(a_n,b_n \neq 0\) for each \(n\in\N\).
Let \(\lambda\) and \(\Lambda\) be the roots of the associated characteristic polynomial \(x^2 -bx - a\) and suppose that \(|\lambda|\neq |\Lambda|\).
Then the  above recurrence has two linearly independent solutions \(\seq{u_n^{(1)}}{n}\) and \(\seq{u_n^{(2)}}{n}\) such that \(u_{n+1}^{(1)}/u_n^{(1)} \sim \lambda\) and \( u_{n+1}^{(2)}/u_n^{(2)} \sim \Lambda\).
\end{theorem}

Later work by Perron \cite{perron1921poincare2} considered the case
that the two roots are equal in modulus, as follows.

\begin{theorem} \label{thm: perron}
Suppose that \(u_n = b_n u_{n-1} + a_n u_{n-2}\) is a Poincar\'{e} recurrence and \(a_n,b_n \neq 0\) for each \(n\in\N\).  
Let \(\lambda\) and \(\Lambda\) be the roots of the associated characteristic polynomial
\(x^2 -bx - a\).
Then the above recurrence has two linearly independent solutions \(\seq{u_n^{(1)}}{n}\) and \(\seq{u_n^{(2)}}{n}\) such that \(\limsup_{n\to\infty} \sqrt[n]{|u_n^{(1)}|} = |\lambda|\) and \(\limsup_{n\to\infty} \sqrt[n]{|u_n^{(2)}|} = |\Lambda|\).
\end{theorem}

We note that one cannot obtain the neat asymptotic equalities of the form given in the Poincar\'{e}--Perron Theorem when the moduli of the roots coincide (consider, for example the Poincar\'{e} recurrence \(u_n = u_{n-2}\) whose characteristic roots are \(\pm 1\)).
However, later work by Kooman gives a complete characterisation of the asymptotic behaviour of linearly independent solutions for a family of second-order Poincar\'{e} recurrence relations. 
We give two illustrating examples illustrating when the characteristic roots
have equal modulus. 
The proof is a straightforward application of
results in \cite{kooman2007asymptotic}. 
We shall make use of these particular forms in the sequal.

\begin{example}[\autoref{app:asymptotics}]\label{ex:KoomanAsymptotics}
\hfil
\begin{enumerate}
	\item The recurrence relation $(n+\alpha) u_n = \beta u_{n-1} + (n + \gamma) u_{n-2}$ with $\alpha$
$\beta$, $\gamma \in \R$ and $\beta > 0$ admits linearly independent solutions
$\seq{u_n^{(1)}}{n}$ and $\seq{u_n^{(2)}}{n}$ with the following asymptotic equalities:
\(u_n^{(1)} \sim
	n^{\frac{1}{2}(\beta + \gamma -  \alpha)}\) and 
\(u_n^{(2)} \sim \mleft(-1 \mright)^n
	n^{\frac{1}{2}(-\beta + \gamma -  \alpha)}\).
	\item The recurrence relation $(n+\alpha) u_n = (2n + \beta) u_{n-1} - (n + \gamma) u_{n-2}$,
with $\alpha$, $\beta$, $\gamma \in \R$ with $\beta > \alpha + \gamma$ admits linearly independent
solutions
$\seq{u_n^{(1)}}{n}$ and $\seq{u_n^{(2)}}{n}$  with the following asymptotic equalities:
\begin{align*}
u_n^{(1)} &\sim  n^{(1-2\alpha +2\gamma)/4} \exp(2\sqrt{(\beta - \alpha - \gamma)n}),  \quad \text{and} \\
u_n^{(2)} &\sim  n^{(1-2\alpha +2\gamma)/4} \exp(-2\sqrt{(\beta - \alpha - \gamma)n}).
\end{align*}
\end{enumerate}
\end{example}

\subsection{Continued fractions}

A \emph{continued fraction} is an ordered pair \(\left((\seq[\infty]{a_n}{n=1}, \seq[\infty]{b_n}{n=0}), \seq[\infty]{f_n}{n=0}\right)\) where \(\seq{a_n}{n}\) and \(\seq{b_n}{n}\) are sequences of complex numbers such that for each \(n\in\N\), \(a_n\neq 0\) and \(\seq{f_n}{n}\) is a sequence in \(\hat{\mathbb{C}} = \mathbb{C}\cup\{\infty\}\) recursively defined by the following composition of linear fractional transformations.
For \(w\in\hat{\C}\), define
\begin{equation*}
	 s_0(w) = b_0 + w \text{ and } s_n(w) = \ \frac{a_n}{b_n + w} \text{ for each } n\in\{1,2,\ldots\}.
\end{equation*}
We set \(f_n := s_0 \circ \cdots \circ s_n(0)\) so that
\begin{equation*}
 f_n = b_0+\cfrac{a_1}{b_1 +\cfrac{a_2}{b_2 +\cfrac{a_3}{\ddots + \cfrac{a_n}{b_n}}}}.
\end{equation*}
It is convenient to introduce concise notation for continued fractions and their convergents.  
We shall make use of Gauss's \text{Kettenbruch} notation  \(f_n =: b_0 + \KF_{k=1}^n  ({a_k}/{b_k})\),
and abuse the infinite form of this notation to refer both to the
continued fractions and to their limits (if they converge).

We respectively call \(\seq{a_n}{n}\) and \(\seq{b_n}{n}\) the sequences of \emph{partial numerators} and \emph{partial denominators} (together the \emph{partial quotients}) of the continued fraction \(\KF (a_n/b_n)\).  
We call \(\seq{f_n}{n}\) the sequence of \emph{convergents}.
Let \(\seq[\infty]{A_n}{n=-1}\) and \(\seq[\infty]{B_n}{n=-1}\) satisfy the recurrence relation \(u_n = b_n u_{n-1} + a_n u_{n-2}\) 
with initial values \(A_{-1}=1, A_0=0, B_{-1}=0,\) and \(B_0=1\).
Then \(\seq{A_n}{n}\) and \(\seq{B_n}{n}\) are respectively called the sequences of \emph{canonical numerators} and \emph{canonical denominators} of \(\KF(a_n/b_n)\) because \(f_n = A_n/B_n\) for each \(n\in\N_0\).
We call \(f\in\hat{\mathbb{C}}\)  the \emph{limit} of the continued fraction \(\KF(a_n/b_n)\) if \(f_n \to f\) as \(n\to\infty\) and say that \(\KF(a_n/b_n)\) \emph{converges} if such a limit exists.
The results presented herein consider continued fractions whose partial quotients are real-valued.
Nevertheless it is often useful to adopt the standard notion of convergence in \(\hat{\mathbb{C}}\) in order to exploit the algebraic properties of \(\hat{\mathbb{C}}\).
%
%
%

Two continued fractions are said to be \emph{equivalent} if they have the same sequence of convergents.
From the standard equivalence transformation (as described in \cite[\S1.4]{cuyt2008handbook} or \cite[Chapter II, Cor.~10]{LW1992}),
We have the following equivalences for
the continued fractions associated to the respective recurrences \eqref{eq: ogformb} and \eqref{eq: sform}: 
		\begin{align*}
				\KF_{n=1}^\infty \frac{g_1(n)/g_3(n)}{g_2(n)/g_3(n)} &\equiv \frac{{g_1(1)}/{g_2(1)}}{1 + \KF_{n=2}^\infty (d_n/1)}, \quad \text{and} \\
				\KF_{n=1}^\infty \frac{g_1(n)g_3(n-1)}{g_2(n)} &\equiv \frac{{g_1(1)g_3(0)}/{g_2(1)}}{1 + \KF_{n=2}^\infty (d_n/1)}.
		\end{align*}	

	%
Here \(d_n = \tfrac{g_1(n)g_3(n-1)}{g_2(n-1)g_2(n)}\) for each \(n\in\N\).
Note that we are permitted to make these transformations under the assumption that \(g_2(n)\neq 0\) for each \(n\in\N\).
It is clear from the tails of the above continued fractions and Pincherle's Theorem (\autoref{thm: pincherle}) that the transformations described between the recurrence forms preserves the existence of minimal solutions. 

A \emph{simple continued fraction} takes the form \(b_0 +\KF_{n=1}^\infty (1/b_n)\) where each partial denominator is a positive integer.
The number \(\pi\) has an erratic simple continued fraction expansion whose  sequence of partial denominators begins \((3; 7, 15, 1, 292, 1, 1, 1, 2, 1, 3, 1,  \ldots)\).
However, Lord Brouncker (as reported by Wallis in \cite{wallis1655arithmetica}\footnote{See the translation by Stedall \cite{wallis2004arithmetic}.}) gave a continued fraction expansion for \(4/\pi\) whose partial quotients are polynomials
as follows:
	\begin{equation*}
		\frac{4}{\pi} = 1 + \KF_{n=1}^\infty \frac{(2n-1)^{2}}{2}.
	\end{equation*}
Ap\'{e}ry's constant \(\zeta(3)\) has a continued fraction expansion (see \cite{poorten1979proof})
	\begin{equation*}
		\zeta(3) = \frac{6}{5 + \KF_{n=1}^\infty ({-n^6}/({34n^3 + 51n^2 + 27n + 5}))}
	\end{equation*}
whose partial quotients are ultimately polynomials.
We refer the reader to \cite{cuyt2008handbook} for further examples of continued expansions of famous constants.
Motivated by such constructions, Bowman and Mc~Laughlin
\cite{bowman2002polynomial} (see also \cite{mclaughlin2005polynomial})
coined the term \emph{polynomial continued fraction}.
A polynomial continued fraction \(\KF (a_n/b_n)\) has integer partial quotients such that for sufficiently large \(n\in\N\) we have \(a_n = p(n)\) and \(b_n = q(n)\) for \(p,q\in\Z[x]\).  
The evaluation of polynomial continued fractions whose partial quotients have low degrees appears in the accounts \cite{perron1957lehre, LW1992, cuyt2008handbook}.
For \(\deg(a_n)\le 2\) and \(\deg(b_n)\le 1\), Lorentzen and Waadeland \cite[ \S6.4]{LW1992} express the polynomial continued fraction \(\KF (a_n/b_n)\) as a ratio of two hypergeometric functions with algebraic parameters.
However, their methods do not cover all cases at low degrees; for example, the polynomial continued fraction \(\KF_{n=1}^\infty \tfrac{-n(n+1)}{2n+1}\) corresponding to the recurrence relation \((n+1)u_n = (2n+1)u_{n-1} - (n+1)u_{n-2}\) cannot be so treated. 
Indeed the method presented in \cite{LW1992} cannot handle cases where the corresponding recurrence has a single repeated characteristic root---the above is one such example with its associated characteristic polynomial \(x^2-2x+1 =(x-1)^2\).

\begin{remark}
Let \(\P\subset\R\) be the set of real numbers that have a polynomial continued fraction expansion.  
We have that \(\Q\subset \P\) and, as can be seen from the literature,
there is a plethora of examples of 
both algebraic and transcendental numbers in \(\P\).
%
%
%
We shall be interested in the problem of determining whether a polynomial continued fraction expansion and an algebraic number are equal.
\end{remark}

\subsection{Convergence criteria for continued fractions}

A continued fraction \(\KF(a_n/b_n)\) is said to be \emph{positive} if \(a_n>0\) and \(b_n\ge 0\) for each \(n\in\N\).%
%

\begin{lemma} \label{prop:continued}
Suppose that for each \(n\in\N\) the sequences \(\seq{a_n}{n}\) and \(\seq{b_n}{n}\) are positive.  
Let \(\seq{f_n}{n}\) be the sequence of convergents associated to the continued fraction \(\KF_{n=1}^\infty (a_n/b_n)\).
Then
	\begin{equation} \label{eq: inequalities}
	f_2 \le f_4 \le \cdots \le f_{2m} \le \cdots \le f_{2m+1} \le \cdots \le f_3 \le f_1.
	\end{equation}
If, in addition, \(b_1>0\) then the subsequences \(\seq{f_{2n}}{n}\) and \(\seq{f_{2n+1}}{n}\) converge to finite, non-negative limits.  If \(b_n>0\) for each \(n\in\N\) then \eqref{eq: inequalities} above holds with strict inequalities.
\end{lemma}
We recall a necessary and sufficient criterion for convergence of a positive continued fraction \cite[Theorem 3.14]{lorentzen2008continued}.
\begin{theorem}[Seidel--Stern Theorem]
\label{th:Stern2}
A positive continued fraction \(\KF(a_n/b_n)\) converges if and only if its \emph{Stern--Stolz series}
	\begin{equation*}
		\mathcal{S}:= \sum_{n=1}^\infty \mleft| b_n \prod_{k=1}^n a_k^{(-1)^{n-k+1}} \mright| 
	\end{equation*}
 diverges to \(\infty\).
\end{theorem}

\subsection{Second-order linear recurrences and continued fractions}

A non-trivial solution \(\seq[\infty]{u_n}{n=-1}\) of the recurrence \(u_n = b_n u_{n-1} + a_n u_{n-2}\) is called \emph{minimal} if there exists another linearly independent solution \(\seq[\infty]{v_n}{n=-1}\) such that \(\lim_{n\to\infty} u_n/v_n = 0\).  
(In such cases the solution \(\seq{v_n}{n}\) is called \emph{dominant}).
If \(\seq{u_n}{n}\) is minimal then all solutions of the form \(\seq{cu_n}{n}\) where \(c\neq0\) are also minimal.
Note that if \(\seq{y_n}{n}\) and \(\seq{z_n}{n}\) are linearly independent solutions of the above recurrence such that \(y_n/z_n \sim C\in\hat{\C}\) then the recurrence relation has a minimal solution \cite{LW1992}.
In general, a system of recurrences may not have a minimal solution.
Nevertheless, if \(\seq{u_n}{n}\) and \(\seq{v_n}{n}\) are respectively minimal and dominant solutions of the recurrence, then together they form a basis of the solution space.

\begin{remark} 
When a second-order recurrence relation has minimal solutions, it is often beneficial from a numerical standpoint to provide a basis of solutions where one of the elements is a minimal solution.
Such a basis can be used to approximate any element of the vector space of solutions:
taking \(\seq{u_n}{n}\) and \(\seq{v_n}{n}\) as above, a general
solution \(\seq{w_n}{n}\) is given by \(w_n = \alpha_1 u_n + \alpha_2
v_n\) and is therefore dominant unless \(\alpha_2=0\).
\end{remark}

Let \(\seq[\infty]{u_n}{n=-1}\) be a non-trivial solution of the recurrence relation \(u_n =b_nu_{n-1} + a_nu_{n-2}\).  If \(u_{n-1}\neq 0\) then we can rearrange the relation to obtain
	\begin{equation*}
		-\frac{u_{n-1}}{u_{n-2}} = \frac{a_n}{b_n - \frac{u_n}{u_{n-1}}}
	\end{equation*}
for \(n\in\N\).
In the event that \(u_{n-2}=0\) we take the usual interpretation in \(\hat{\C}\).
Since \(\seq{u_n}{n}\) is non-trivial  and \(a_n\neq 0\) for each \(n\in\N\), the sequence \(\seq{u_n}{n}\) does not vanish at two consecutive indices.
Thus if \(u_{n-1}=0\) then \(u_{n-2},u_n\neq 0\) and so both the left-hand the right-hand sides of the last equation are well-defined in \(\hat{\C}\) and are equal to \(0\).  Thus the sequence with terms \(-u_n/u_{n-1}\) is well-defined in \(\hat{\C}\) for each \(n\in\N_0\).

The next theorem due to Pincherle \cite{pincherle1894} connects the existence of minimal solutions for a second-order recurrence to the convergence of the associated continued fraction (see also \cite{gautschi1967computational, LW1992, cuyt2008handbook}).
\begin{theorem}[Pincherle] \label{thm: pincherle}
Let \(\seq[\infty]{a_n}{n=1}\) and \(\seq[\infty]{b_n}{n=1}\) be real-valued sequences such that each of the terms \(a_n\) is non-zero.  
First, the recurrence \(u_n = b_n u_{n-1} + a_n u_{n-2}\) has a minimal solution if and only if the continued fraction \(\KF ({a_n}/{b_n})\) converges.
Second, if \(\seq{u_n}{n}\) is a minimal solution of this recurrence 
then the limit of \(\KF ({a_n}/{b_n})\) is \(-u_0/u_{-1}\).
As a consequence, the sequence of canonical denominators \(\seq[\infty]{B_n}{n=-1}\) is a minimal solution if and only if the value of \(\KF ({a_n}/{b_n})\) is \(\infty\in\hat{\mathbb{C}}\).
\end{theorem}

We refer to the problem of determining whether the value a given convergent polynomial
continued fraction is equal to a particular algebraic number as the
\emph{PCF Equality Problem}. We now have:

\begin{corollary}[\autoref{app:PCFmin}]
\label{cor:PCFmin}
The PCF Equality Problem and the Minimality Problem are interreducible.
\end{corollary}

We denote by \(\Q(x)\) the field of rational functions; that is, the field of fractions of the polynomial ring \(\Q[x]\).  We define the \emph{degree} of \(r=p(x)/q(x)\in\Q(x)\) as follows: if \(r=0\) set \(\deg(r)=-\infty\), otherwise set \(\deg(r)=\deg(p) - \deg(q)\).
%
%
%
%

The following theorem relates the convergence of the polynomial continued fraction \(\KF_{n=1}^\infty ({p(n)}/{q(n)}) \) to the behaviour of an associated rational function \cite{kooman1990convergence} (see also the version of the theorem presented in \cite{kooman1991convergence} for the field of meromorphic fractions).
	\begin{theorem} \label{thm:contFracConvChar}
		For \(p,q\in\Q[n]\) such that neither \(p\) nor \(q\) is the zero polynomial, let \(r\in\Q(n)\) be the rational function given by \(r(n) = 1 + \tfrac{4q(n)}{p(n-1) p(n)}\) with \(\deg(r)=d\).
		The continued fraction \(\KF_{n=1}^\infty ({p(n)}/{q(n)})\) converges if and only if one of the following holds:
		\begin{enumerate}
			\item \(\deg(r) \le -2\) \text{ and } \(\lim_{n\to\infty} r(n) n^2 \ge -1/4\), or
			\item \(-1\le \deg(r) \le 2\) \text{ and } \(\lim_{n\to\infty} r(n) n^{d} >0\).
		\end{enumerate}	
	\end{theorem}
We remark the immediate corollary by \autoref{thm: pincherle}.
\begin{corollary}\label{cor:minimalityDecidable}
Given a recurrence relation of the form \eqref{eq: ogform}, it is decidable whether the recurrence admits a minimal solution.
\end{corollary}

%

The following technical lemma is well-known (see, for example, \cite[Lemma 4, \S IV]{LW1992}).
\begin{lemma} \label{lem: solutions}
	Suppose that \(\seq{u_n}{n}\) and \(\seq{v_n}{n}\) are both solutions to the recurrence relation \(u_n = b_nu_{n-1} + a_nu_{n-2}\).  Then
		\begin{equation*}
			u_n v_{n-1} - u_{n-1} v_n = (u_0 v_{-1} - u_{-1} v_0)\prod_{k=1}^n (-a_k).
		\end{equation*}
\end{lemma}

\begin{remark}
Given a \(k\)th-order recurrence relation \(R\) with coefficients in
\(\Q(x)\), let \(Z(R)\) be the set of solution sequences with rational initial values.
For \(\seq{u_n}{n},\seq{v_n}{n}\in Z(R)\), consider the limit \(u_n/v_n\) as \(n\to\infty\) if the limit exists and let \(L(R)\) be the set of such limits
	\begin{equation*}
		L(R) := \mleft\{\alpha\in\R : \alpha = \lim_{n\to\infty} \frac{u_n}{v_n},\, \seq{u_n}{n},\seq{v_n}{n}\in Z(R)	\mright\}.
	\end{equation*}
Because \(Z(R)\) is a vector space over \(\Q\), it follows that \(\Q\subset L(R)\subset \R\). 
Let \(\mathbb{L}\) be the union of \(L(R)\) over all \(R\).
Kooman \cite[Chapter 2]{kooman1991convergence} makes the following observations:
the set \(\mathbb{L}\) is a field, is countable, and \(\overline{\mathbb{Q}}\cap\R \subset \mathbb{L} \subset \R\).  
We note the inclusion \(\overline{\mathbb{Q}}\cap\R \subset
\mathbb{L}\) follows from limits associated to \(C\)-finite recurrence relations.
The set \(\mathbb{L}\) also contains real transcendental numbers.
In fact, any real number of the form \(\sum_{k=0}^\infty \prod_{m=1}^k q_m\) with \(q_m\in\Q(m)\) such that \(q_m, 1/q_m \neq 0\) is a limit of a solution to the second-order recurrence \(u_n = (1+q_n)u_{n-1} - q_n u_{n-2}\).
We connect such limits to minimal solutions of second-order recurrence relations in the next remark.
\end{remark}

\begin{remark}\label{rem:All1Sol}
Consider the recurrence relation
	\begin{equation} \label{eq: exform}
		u_n = (1+q_n)u_{n-1} - q_n u_{n-2}.  
	\end{equation}
The constant sequence \(\seq[\infty]{1}{n=-1}\) is clearly a solution to the recurrence. 
By \autoref{lem: solutions}, we obtain the second solution \(\seq[\infty]{v_n}{n=-1}\) with initial terms \(v_{-1}=0\), \(v_0=1\), and for \(n\in\N\), \(v_n = \sum_{k=0}^n \prod_{m=1}^k q_m\) where the empty product is equal to unity.  These two solutions are linearly independent and it is interesting to ask whether the above recurrence relation has a minimal solution.

Let \(\xi := \lim_{n\to\infty} v_n/u_n = \sum_{k=0}^\infty \prod_{m=1}^k q_m\) if the limit exists.
We have the following characterisation for minimal solutions in terms of \(\xi\).
If \(\xi=\infty\) then \(\seq{u_n}{n}\) is a minimal solution of \eqref{eq: exform}.
If \(\xi\in\R\) then consider the non-trivial sequence \(\seq{w_n}{n}\) with terms \(w_n = v_n - \xi u_n\).  
Clearly \(\lim_{n\to\infty} w_n/u_n = 0\) and so we conclude that \(\seq{w_n}{n}\) is a minimal solution.  As a side note in the case that \(\xi=0\), \(\seq{v_n}{n}\) is a minimal solution.
\end{remark}

\begin{example}\label{ex:criticalRatioValues}
A series \(\sum c_kx^k\) is called \emph{hypergeometric} if the ratio of consecutive summands \(c_{k+1}/c_k\) is equal to a rational function of \(k\) for each \(k\in\N_0\).
It can be shown (see \cite{andrews1999special}) that a hypergeometric series can be written as follows
	\begin{equation*}
		\sum_{k=0}^\infty c_k = c_0 \sum_{k=0}^\infty \frac{(\alpha_1)_k \cdots (\alpha_j)_k}{(\beta_1)_k \cdots (\beta_\ell)_k} \frac{x^k}{k!} =: c_0 \hypgeo{\mathnormal{j}}{\ell}(\alpha_1,\ldots,\alpha_j; \beta_1,\ldots, \beta_{\ell}; x).
	\end{equation*}
%
For $\alpha \in \C$ the \emph{(rising) Pochhammer symbol} $(\alpha)_n$ is defined as
$(\alpha)_0 = 1$, and $(\alpha)_n = \prod_{j = 0}^{n-1}(\alpha + j)$ for $n\geq 1$.
Here the parameters \(\beta_m\) are not negative integers or zero for otherwise the denominator would vanish for some \(k\).
It is useful in the sequel (\autoref{prop:111repeatedRoot}) to connect hypergeometric series and the recurrence relation in \autoref{rem:All1Sol}.
If we choose
\begin{equation*}
q_m = \frac{(\alpha_1+m-1)\cdots (\alpha_j+m-1)}{(\beta_1+m-1)\cdots (\beta_{\ell}+m-1)m}x
\end{equation*}
in order that \(q_m,1/q_m\neq 0\) for each \(m\in\N\),
then 
	\begin{equation*}
		\xi = \sum_{k=0}^\infty \prod_{m=1}^k q_m = \hypgeo{\mathnormal{j}}{\ell}(\alpha_1,\ldots,\alpha_j; \beta_1,\ldots, \beta_{\ell}; x).
	\end{equation*}
\end{example}

\section{Relations Between Oracles  for Holonomic Sequences}
\label{sec:posmin}

In this section, we examine how the problems of Minimality,
Positivity, and Ultimate Positivity\footnote{The \emph{Ultimate
    Positivity Problem} asks whether a holonomic sequence takes on
  non-negative values for all but finitely many terms.} for
second-order holonomic sequences relate to each
other. We shall assume throughout that each of the polynomial
coefficients in the associated recurrence~\eqref{eq: ogform} has
degree at most \(1\).
It is convenient to introduce the following notation:
we set $g_3(n) = \alpha_1 n + \alpha_0$, $g_2(n) = \beta_1 n + \beta_0$, and,
$g_1(n) = \gamma_1 n + \gamma_0$.
Thus the recurrence relation under focus is of the form
\begin{equation}\label{eq:lowDegreeOGForm}
(\alpha_1 n + \alpha_0) u_{n} = (\beta_1 n + \beta_0) u_{n-1} + (\gamma_1 n + \gamma_0)u_{n-2}.
\end{equation}

We have the following results.

\begin{theorem}
\label{th:oracleRel}
The following hold for the family of holonomic sequences satisfying second-order recurrences of degree at most one.
\begin{enumerate}
\item The Positivity Problem reduces to the Minimality Problem.
\item The Positivity Problem and the Ultimate Positivity
  Problems are interreducible (see \autoref{app:Ultpos}).
%
\end{enumerate}
\end{theorem}

%
%


The rest of this section is devoted to the proof of the first item of~\autoref{th:oracleRel}.
It is useful to separate the problem into subcases according to the signature of the recurrence relation \(u_n = b_n u_{n-1} + a_n u_{n-2}\) with $b_n= g_2(n)/g_3(n)$ and $a_n=g_1(n)/g_3(n)$.
When the signature of the recurrence is either \((+,+)\) or \((-,-)\) then the problem of deciding whether a solution sequence \(\seq{u_n}{n}\) with initial terms \(u_{-1},u_0\ge 0\) is trivial.
If the recurrence has signature \((+,+)\) then \(\seq{u_n}{n}\) is positive, whilst if the recurrence has signature \((-,-)\) then \(u_1 < 0\) and so the solution sequence is not positive.
%
It remains to consider the cases \((-,+)\) and \((+,-)\).
Recall the canonical solutions \(\seq[\infty]{A_n}{n=-1}\) and \(\seq[\infty]{B_n}{n=-1}\) defined in the preliminaries.
For a recurrence relation with signature \((-,+)\), the canonical solutions \(\seq[\infty]{A_n}{n=-1}\) and \(\seq[\infty]{B_n}{n=-1}\) have terms \(A_2<0\) and \(B_1=b_1 <0\).
Thus in our discussion of the Positivity Problem for non-trivial solutions we can assume that \(u_{-1},u_0 >0\).
For a recurrence relation with signature \((+,-)\) we have that \(A_1=a_1<0\) and so one can assume that \(u_0>0\).
We defer our discussion of the positive terms in the solution sequence \(\seq[\infty]{B_n}{n=-1}\) until later in this section.

We shall treat the two signatures separately. We shall first handle recurrence relations with signature $(-,+)$. In this case we have the following
result.
%

\begin{proposition}
\label{prop:pos}
Suppose that \(\seq{u_n}{n}\) with initial values \(u_{-1},u_0>0\) is a solution sequence for recurrence \eqref{eq:lowDegreeOGForm} with signature \((-,+)\)
and the associated continued fraction \(\KF(a_n/b_n)\) converges to a finite limit \(\mu\).
 The following statements are equivalent:
\begin{enumerate}
\item the sequence \(\seq{u_n}{n}\) is positive,
\item the sequence \(\seq{u_n}{n}\) is minimal, and 
\item $-u_0/u_{-1}=\mu$.
\end{enumerate}
\end{proposition}

For the proof, we need the next lemma, which
links positivity of a solution sequence $\seq{u_n}{n}$ to the sequence of the convergents
\(\seq{f_n}{n}\) of the continued fraction \(\KF(a_n/b_n)\).

\begin{lemma}
\label{lem: ctdfractail}
Suppose that \(\seq{u_n}{n}\) is a solution sequence for recurrence \eqref{eq:lowDegreeOGForm} with signature \((-,+)\).
Assume that \(u_{-1}>0\).
For even \(n\in\N\), \(u_n>0\) if and only if \(f_n > -u_0/u_{-1}\).
For odd \(n\in\N\), \(u_n>0\) if and only if \(f_n < -u_0/u_{-1}\).
\end{lemma}
\begin{proof}
For the canonical solution sequences \(\seq[\infty]{A_n}{n=-1}\) and \(\seq[\infty]{B_n}{n=-1}\) we have that \(u_n = A_n u_{-1} + B_n u_{0}\) for each \(n\in\{-1,0,\ldots\}\).
For recurrences with signature \((-,+)\), it is easy to show by induction that \(B_n<0\) for each odd \(n\in\N\), and \(B_n>0\) for each even \(n\in\N\).
Thus for even \(n\in\N\) we have that \(u_n>0\) if and only if \(A_n/B_n + u_0/u_{-1} = f_n + u_0/u_{-1} > 0\).
The case for odd \(n\in\N\) is handled in the same fashion.
\end{proof}

Equipped with the above observation, we are in the position to conclude \autoref{prop:pos}.

\begin{proof}[Proof of \autoref{prop:pos}]
We note that \(\seq[\infty]{-f_n}{n=1}\) is the sequence of convergents associated with the positive continued fraction \(\KF_{n=1}^\infty \tfrac{a_n}{-b_n}\).
By \autoref{prop:continued}, the subsequences 
\(\seq[\infty]{-f_{2n}}{n=1}\) and \(\seq[\infty]{-f_{2n-1}}{n=1}\)
converge to finite limits $-\ell_1$ and $-\ell_2$, respectively. 
By~\autoref{lem: ctdfractail}, a solution sequence $\seq{u_{n}}{n}$ is positive if and only if $\ell_2 \leq -u_0/u_{-1} \leq \ell_1$.
The Stern--Stolz series (from \autoref{th:Stern2}) associated to the continued fraction \(\KF_{n=1}^\infty \tfrac{a_n}{-b_n}\) diverges due to our assumption that each of the coefficients in \eqref{eq:lowDegreeOGForm} is a polynomial with degree in \(\{0,1\}\).
We conclude that $\ell_1 = \ell_2$ by \autoref{th:Stern2}.
Thus \(\seq{u_n}{n}\) is positive if and only if $-u_0/u_{-1}$ is equal to $\mu=\ell_1=\ell_2$.
From \autoref{thm: pincherle}, a solution sequence \(\seq{u_n}{n}\) is minimal if and only if \(-u_0/u_{-1}\) is the value of the continued fraction \(\KF_{n=1}^\infty (a_n/b_n)\).
\end{proof}

We now consider recurrences \(u_n = b_n u_{n-1} + a_n u_{n-2}\) with signature \((+,-)\).
Given our restriction on the degrees of the polynomial coefficients, we can assume, without loss of generality,
that the sequences of coefficients  $\seq{b_n}{n}$ and $\seq{a_n}{n}$ are monotonic.
In the work that follows we split the \((+,-)\) case into two further subcases depending on the sign of the discriminant of the recurrence
$\Delta(n):= b_n^2+4a_n$. 
We shall assume that \(\sign(\Delta(n))\) is constant as this can be achieved by a suitable computable shift of the recurrence relation.
%
The discussion of the subcase \(\Delta(n)<0\) is given in 
\autoref{app:ComplexCharacteristicRoots}.
Let us summarise the results established therein with the following proposition.
\begin{proposition}\label{prop:noPositiveMinimalSolution}
A recurrence relation of the form \eqref{eq:lowDegreeOGForm} with signature \((+,-)\) and discriminant $\Delta(n) < 0$ for each $n\in \mathbb{N}$ has no positive non-trivial solutions.
\end{proposition}
Let us turn our attention to the subcase \(\Delta(n)\ge 0\). 
We first need some technical lemmas, the first of which shows that such a recurrence relation admits a non-trivial positive solution.

\begin{lemma} \label{lem:Bpositive}
Consider a normalised recurrence \(u_n = b_nu_{n-1} + a_nu_{n-2}\) with signature \((+,-)\) such that \(\Delta(n)\ge 0\) for each \(n\in\N\).
Let \(\seq[\infty]{B_n}{n=-1}\) be the canonical solution sequence with initial conditions \(B_{-1}=0\) and \(B_0=1\) associated to this recurrence.
Then for each \(n\in\N\), \(B_n > 0\).
\end{lemma}
\begin{proof}
We separate the proof into two cases depending on the monotonicity of the coefficients \(\seq[\infty]{a_n}{n=1}\).
Let us suppose that \(\seq{a_n}{n}\) is increasing.
It is sufficient to show that \(B_n / B_{n-1} \ge \sqrt{-a_n}\) as
\(B_1 = b_1 \ge 2\sqrt{-a_1} > 0\).
For the induction step, we have the inequalities below using our assumptions on the discriminant and the monotonicity of \(\seq{a_n}{n}\):
	\begin{equation*}
		B_n / B_{n-1} = b_n + a_n B_{n-2}/ B_{n-1} \ge 2\sqrt{-a_n} + a_n/\sqrt{-a_{n-1}} \ge \sqrt{-a_n}.
	\end{equation*}

Now suppose that \(\seq{a_n}{n}\) is decreasing.
Consider the recurrence sequence \(\seq[\infty]{v_n}{n=-1}\) with terms \(v_{-1}=0\), \(v_0=1\), and for \(n\in\N\), \(v_n = (-1)^n u_n / \prod_{k=1}^{n+1} g_1(k)\).
The sequence \(\seq{v_n}{n}\) satisfies the recurrence \(v_n = b_n' v_{n-1} + a_n' v_{n-2}\) with coefficients \(b_n' = -g_2(n)/(g_1(n+1)g_3(n))\) and \(a_n' = 1/(g_1(n+1)g_3(n))\), and signature \((+,-)\).
Clearly \(B_n>0\) for each \(n\in\N\) if and only if \(v_n >0\) for each \(n\in\N\).
By assumption, \(\seq[\infty]{a_n'}{n=1}\) is an increasing sequence and so we have that \(\seq{v_n}{n}\) is a positive sequence from the previous case.
Note the above transformation does not preserve the degrees of coefficients in the recurrence relation.
However, the induction proof above does not depend on the degrees of the polynomial coefficients in the recurrence relation.
\end{proof}

\begin{lemma} \label{lem:fndecreasing}
Suppose that recurrence \eqref{eq:lowDegreeOGForm} has discriminant \(\Delta(n)\ge 0\) and signature \((+,-)\).
Then the sequence of convergents \(\seq{f_n}{n}\) associated with the continued fraction \(\KF(a_n/b_n)\) is strictly decreasing.
\end{lemma}
\begin{proof}
By \autoref{lem:Bpositive}, \(B_n > 0\) for each \(n\in\N\).
From \autoref{lem: solutions} we have that \(A_n B_{n-1} - A_{n-1} B_n = -\prod_{k=1}^n (-a_k)\).
Thus
	\begin{equation*}
		f_n - f_{n-1} = \frac{A_n}{B_n} - \frac{A_{n-1}}{B_{n-1}} = -\frac{\prod_{k=1}^n (-a_k)}{B_{n-1}B_{n}} < 0,
	\end{equation*}
and so \(\seq{f_n}{n}\) is strictly decreasing.
\end{proof}

We again link the positivity of a solution to the sequence of convergents of the associated continued fraction using the lemma that follows.
\begin{lemma} \label{lem:positivity}
Suppose that \(\seq[\infty]{u_n}{n=-1}\) is a solution of the normalised recurrence \(u_n = b_nu_{n-1} + a_nu_{n-2}\) with signature \((+,-)\) such that \(\Delta(n)\ge 0\) for each \(n\in\N\).
Assume that \(u_{-1}>0\).  
Given \(N\in\N\), we have that \(-u_0/u_{-1} < f_N\) if and only if \(u_{N}>0\).
\end{lemma}
\begin{proof}
For each \(N\in\N\), \(u_N = u_{-1} A_N +  u_0 B_N\) where \(\seq[\infty]{A_n}{n=-1}\) and \(\seq[\infty]{B_n}{n=-1}\) are the canonical solutions.
It follows that \(-u_0/u_{-1} < A_N/B_N = f_N\) if and only if \(u_N>0\).
Here we have used the assumption that \(u_{-1}>0\) and \(B_N > 0\) (from \autoref{lem:Bpositive}).
\end{proof}

We are now in the position to characterise positive solutions to the considered recurrence relations via the ratio of the initial terms.
\begin{proposition} \label{prop:mono}
Suppose that \(\seq[\infty]{u_n}{n=-1}\) is a solution of recurrence \eqref{eq:lowDegreeOGForm} with signature \((+,-)\) such that \(\Delta(n)\ge 0\) for each \(n\in\N\).
First, the associated continued fraction \(\KF_{n=1}^\infty (a_n/b_n)\) converges to a finite limit 
\(\mu < 0\).
Second, a solution \(\seq[\infty]{u_n}{n=-1}\) with \(u_{-1},u_0>0\) is positive if and only if \(-u_0/u_{-1} \le \mu\).
\end{proposition}

\begin{proof}
Since the sequence of convergents \(\seq{f_n}{n}\) associated with the continued fraction 
\(\KF_{n=1}^\infty (a_n/b_n)\) is strictly decreasing (by \autoref{lem:fndecreasing}), it is clear the limit exists.
Since \(f_1 = a_1/b_1 < 0\), if the value \(\mu\) is finite then \(\mu<0\).

We claim that $\mu$ is finite.  Subject to this assumption, let \(\seq[\infty]{u_n}{n=-1}\) be a solution to the 
recurrence relation.
For each \(N\in\N_0\), \(u_{-1}, u_0,\ldots, u_{N}>0\) if and only if \(-u_0/u_{-1} < f_N\) 
by \autoref{lem:positivity}.
Thus \(\seq[\infty]{u_n}{n=-1}\) is a positive solution if and only if \(-u_0/u_{-1} \le \mu\).

Let us prove that, subject to our assumptions, $\mu$ is indeed finite. 
Suppose, for a contradiction, that $\mu$ is infinite. 
As we have assumed that for each $\ell\in\{1,2,3\}$, $g_\ell(0)\neq 0$, we can define the recurrence corresponding to a one-step backward shift and extend uniquely any given sequence \(\seq[\infty]{u_n}{n=-1}\) to a sequence \(\seq[\infty]{u_n}{n=-2}\). 
It follows from the recursive definition of the sequence of convergents that if \(\KF_{n=1}^{\infty}(a_n/b_n)\) converges to $\infty$, then \(\KF_{n=0}^{\infty}(a_n/b_n)\) converges to $0$.
This conclusion is not possible as the sequence of convergents is strictly decreasing and \(f_0 = a_0/b_0 = g_1(0)/g_2(0) < 0\).
\end{proof}

We combine the results in Proposition~\ref{prop:mono} and \autoref{thm: pincherle}
into the following corollary.  
%
%
\begin{corollary}\label{cor:MinimalMuPositiveMu}
Let \(\seq[\infty]{u_n}{n=-1}\) be a solution of recurrence relation \eqref{eq:lowDegreeOGForm} with signature \((+,-)\) and \(\Delta(n)\ge 0\) for each \(n\in\N\).
Then a solution sequence $\seq{u_n}{n}$ with \(u_{-1},u_0>0\) is positive if and only if $-{u_0}/{u_{-1}} \le \mu$.
In addition, if $-{u_0}/{u_{-1}} = \mu$ then the sequence \(\seq{u_n}{n}\) is a minimal solution.
\end{corollary}

The difficulty one encounters when determining positivity arises when \(-u_0/u_{-1}\) is equal to the value \(\mu\) of the associated continued fraction.

%
%
\begin{proposition}[Proof in \autoref{ap:compxtomu}] \label{prop:compxtomu}
Let \(\seq[\infty]{u_n}{n=-1}\) be a non-trivial solution sequence for recurrence \eqref{eq:lowDegreeOGForm} 
with signature \((+,-)\) and suppose that $\Delta(n) \ge 0$ for each $n\in \mathbb{N}$.
Then one can detect if 
$-{u_0}/{u_{-1}} < \mu$.
\end{proposition}

We deduce that if one can decide whether a holonomic sequence \(\seq{u_n}{n}\) that solves recurrence \eqref{eq:lowDegreeOGForm} is minimal, then one can decide whether \(\seq{u_n}{n}\) is a positive solution.
\begin{proof}[Proof of \autoref{th:oracleRel}(1)]
Assume we have an oracle for the Minimality Problem for solutions $\seq[\infty]{u_n}{n=-1}$ to
recurrences of the form \eqref{eq:lowDegreeOGForm}.
Given such a recurrence, the existence of a positive solution
is decidable by combining \autoref{prop:pos}, \autoref{prop:noPositiveMinimalSolution}, and \autoref{cor:MinimalMuPositiveMu}. 
We may thus focus on instances where the associated recurrence relation
admits positive solutions. 
Notice that \autoref{prop:pos} and \autoref{cor:MinimalMuPositiveMu}
imply the existence of minimal solutions.
A trivial solution is straightforward to detect.
If $\seq{u_n}{n}$ is minimal, then it is positive by \autoref{prop:pos} and
\autoref{cor:MinimalMuPositiveMu}. Assume now that the sequence is dominant. If the signature of the associated
recurrence relation is $(-,+)$, then the sequence is not positive by \autoref{prop:pos}. 
Assume
then that the signature is $(+,-)$. 
By \autoref{prop:compxtomu}, one can detect if $-u_0/u_{-1} < \mu$.
The case \(-u_0/u_{-1}>\mu\) can also be detected as the sequence contains a negative term. 
This process is equivalent to deciding whether $\seq{u_n}{n}$ is positive.
\end{proof}

\section{Minimality for Degree-1 Holonomic Sequences}
Recall from \autoref{cor:minimalityDecidable} that the problem of whether a recurrence relation
of the form \eqref{eq: ogform} admits a minimal solution is
decidable. 
In the present section, we focus on
the Minimality Problem for such recurrences.
For ease of notation, we parametrise the problem as follows.
\begin{problem}[$\textsc{Minimality}(j,k,\ell)$]
Given a solution $\seq[\infty]{v_n}{n=-1}$ to \eqref{eq: ogform} with $\deg(g_3) = j$, $\deg(g_2) = k$,
and $\deg(g_1) = \ell$, decide whether $\seq{v_n}{n}$ is
minimal.
\end{problem}
Problem $\textsc{Minimality}(0,0,0)$ asks one to determine whether a holonomic sequence that solves a
second-order \(C\)-finite recurrence is a minimal solution. 
Notice that
this is a special case of $\textsc{Minimality}(1,1,1)$ (multiply each of the coefficients
by $(n+1)$) and is therefore not treated separately in the sequel.
In this section we are interested in the decidability of $\textsc{Minimality}(j,k,\ell)$
subject to the restriction that $j,k,\ell \leq 1$.
The main result of this section is the following.
\begin{theorem}\label{thm:MinimalityRedToPeriodLike}

For $j,k,\ell \leq 1$, $\textsc{Minimality}(j,k,\ell)$ reduces to determining whether a period,
an exponential period, or a period-like integral is equal to zero.
\end{theorem}
For definitions and discussion of periods, exponential periods, and period-like integrals
see \autoref{app:Conjectures}.

Observe that the cases for which some of the coefficient polynomials are identically $0$ are dealt with in Subsection \ref{sub:linrec}.
Hence, throughout this section, we assume that $j,k,\ell \in \{0,1\}$, i.e., none of the coefficient polynomials are identically $0$. 
We thus focus on recurrences of the form \eqref{eq:lowDegreeOGForm}, and establish the
following conventions. 
In the case that $\deg(g_3) = 0$, we understand that $\alpha_1 = 0$. 
A
similar convention is applied to the polynomials $g_2$ and $g_1$. On the other hand, we shall
always assume that the values $\alpha_0$, $\beta_0$, and $\gamma_0$ are non-zero in accordance
with the assumption that the polynomial coefficients do not vanish on non-negative integers.
One further assumption is made: the recurrence relations considered in this section are assumed
to admit minimal solutions. This is no loss of generality, as this is
a decidable property, as per Corollary~\ref{cor:minimalityDecidable}.

The proof of \autoref{thm:MinimalityRedToPeriodLike} is spread over several subsections with intermediate results. 
On the face of it,
we have eight different problems to consider. We first reduce the number of problems to five by
establishing some interreductions between the problems $\textsc{Minimality}(j,k,\ell)$ for
different values of the parameters $j$, $k$, $\ell$. We further employ minimality-preserving
transformations to obtain certain canonical instances of each of the remaining problems (\autoref{cor:minimalityReductionForms}).
These are then showed to reduce to checking whether a period-like integral vanishes.
For four of the cases we analyse an associated \emph{generating function} that
connects our sequences to the theory of differential equations.
The conclusion of the statement can be can be pieced together from
\autoref{prop:minimalityRepeatedRoots}, \autoref{prop:minimality111_110_101}, and \autoref{subsec:Minimality010}.

\subsection{Interreductions of \texorpdfstring{{Minimality}$(j,k,\ell)$}{\textsc{Minimality}(j,k,l)}}
\label{subsec:interreductionsLowDegrees}
In this subsection we establish some interreductions of the Minimality Problem for
degree-1 holonomic sequences. We also identify some canonical instances on which we focus
thereafter.

\begin{proposition}\label{prop:linearReductions}
\hfil 
\begin{enumerate}

\item $\textsc{Minimality}(0,k,1)$ reduces to $\textsc{Minimality}(1,k,0)$ and vice versa.

\item
$\textsc{Minimality}(1,0,0)$ reduces to $\textsc{Minimality}(1,1,1)$.

\item $\textsc{Minimality}(1,1,1)$ reduces to the Minimality Problem for solutions to a recurrence of the form
\begin{equation}
\label{eq: lowDegreeForm111}
(n+\alpha)u_n = (\beta_1 n + \beta_0)u_{n-1} + (\gamma_1 n + \gamma_0)u_{n-2},
\end{equation}
where the coefficients are elements of $\Q[n]$, $\beta_1 > 0$ and $|\gamma_1| = \beta_1$. In
the case that $\beta_1^2 + 4\gamma_1 = 0$ in the above recurrence, then  the problem further reduces to the Minimality
Problem for solutions to a recurrence of the form
\begin{equation}\label{eq: lowDegreeForm111RepeatedRoots}
(n+\alpha) u_n = (2n + \beta_0) u_{n-1} - (n+\gamma_0) u_{n-2},
\end{equation}
where the coefficients are in $\Q[n]$.

\item $\textsc{Minimality}(1,1,0)$ reduces to the Minimality Problem for solutions to a recurrence of the
form
\begin{equation}
\label{eq: lowDegreeForm110} (n+\alpha)u_n = (\beta_1 n + \beta_0)u_{n-1} + \gamma_0 u_{n-2},
\end{equation}
where the coefficients are in $\Q[n]$, $\beta_1 > 0$ and $|\gamma_0| = \beta_1$.

\item $\textsc{Minimality}(1,0,1)$ reduces to the Minimality Problem for solutions to
a recurrence of the form
\begin{equation}
\label{eq: lowDegreeForm101} (n+\alpha)u_n = \beta_0 u_{n-1} + (n+\gamma_0) u_{n-2},
\end{equation}
where $\alpha_0, \gamma_0 \in \Q$ and $\beta_0 \in \overline{\mathbb{Q}} \cap \R_{>0}$.
\end{enumerate}
\end{proposition}
Notice that sequences satisfying \eqref{eq: lowDegreeForm101} are not necessarily holonomic (as defined in this note), as we allow the coefficient $\beta_0$ to be irrational. We note that
the reduction of $\textsc{Minimality}(1,0,1)$ to \eqref{eq: lowDegreeForm101} 
yields values of the 
parameter $\beta_0$ that are real algebraic numbers of degree at most $2$.

\begin{proof} 
We prove Item 2 below. Item 1 follows immediately from the interreductions of the recurrence relations \eqref{eq: ogform} and \eqref{eq: sform} described in the preliminaries.  The other items follow from minimality preserving transformations
of the form $\seq{u_n}{n} \mapsto \seq{\kappa^n u_n}{n}$ for some appropriate constants
$\kappa$. In the last item we need to know that the recurrence admits a minimal solution if and only if $\alpha_1 \gamma_1 > 0$ (compare to \autoref{lem:whenHasMinimalSolution}). See \autoref{app:interreductions} for the complete proofs.

%
%
%
Let us normalise the recurrence relation by dividing through by the leading coefficient of $g_3$.
The normalised recurrence is given by
\begin{equation}\label{eq: 100}
(n + \alpha)u_n = \beta u_{n-1} + \gamma u_{n-2}
\end{equation}
with coefficients in \(\Q[n]\).  
We can assume $\alpha := \alpha_0/\alpha_1 > 1$ by taking an appropriate shift.

Suppose that $\seq{u_n}{n}$ is a solution to \eqref{eq: 100}. 
We observe that $\seq{u_n}{n}$ satisfies the recurrence
\begin{equation}\label{eq:modifiedRecurrence}
(n+\alpha - 1)(n+\alpha) u_n = (2\gamma n + \beta^2 + \gamma( 2\alpha - 3))u_{n-2} - \gamma^2 u_{n-4}
\end{equation}
for $n\geq 3$.
This observation follows easily by substituting  $\beta u_{n-3} = (n+\alpha - 2) u_{n-2} - \gamma u_{n-4}$ performing straightforward algebraic manipulations:
\begin{align*}
(n+\alpha - 1)(n+\alpha) u_n &= (n + \alpha -1)(\beta u_{n-1} + \gamma u_{n-2})\\
	&= (2\gamma n + \beta^2 + \gamma( 2\alpha - 3))u_{n-2} - \gamma^2 u_{n-4}.
\end{align*}

Let us now define the sequence $\seq[\infty]{z_n}{n=0}$ by $z_n = u_{2n}$ for each $n\in\N_0$.
Then
$z_0 = u_0$ and
\begin{equation*}
z_1 = u_2 = \frac{1}{\alpha+2} (\beta u_1 + \gamma u_0)  = \mleft(\frac{\beta^2}{(\alpha+2)(\alpha+1)}+\gamma \mright)u_0 + \frac{\beta \gamma}{(\alpha + 2)(\alpha + 1)}u_{-1}.
\end{equation*}
It is clear that for $n\geq 2$
\begin{equation}\label{eq:modifiedRecurrence2}
(2n+\alpha - 1)(2n+\alpha) z_n = (4\gamma n + \beta^2 + \gamma( 2\alpha - 3)) z_{n-1} - \gamma^2 z_{n-2}
\end{equation}
by \eqref{eq:modifiedRecurrence} with the mapping $n\mapsto 2n$.
This establishes a one-to-one correspondence between the solutions to
\eqref{eq:modifiedRecurrence2} and the solutions to \eqref{eq: 100}.
We claim that this equivalence of solutions preserves minimality.
%
%
To conclude the claim, observe that a solution to \eqref{eq:modifiedRecurrence2} can be transformed into a solution to the recurrence
\begin{equation*}
(2n+\alpha - 1) v_n = (4\gamma n + \beta^2 + \gamma( 2\alpha - 3)) v_{n-1} - \gamma^2 (2n+\alpha-2)v_{n-2}
\end{equation*}
such that this transformation preserves minimality (using the same minimality preservation reduction from \eqref{eq: ogform} to \eqref{eq: sform}).

Let us prove that minimality is preserved as claimed.
We first show that the solutions $\seq{u_n}{n}$ and $\seq{v_n}{n}$ to \eqref{eq: 100} are
linearly independent if and only if the solutions $\seq{u_{2n}}{n}$ and $\seq{v_{2n}}{n}$ to
\eqref{eq:modifiedRecurrence2} are linearly independent.
One direction is trivial: if $\seq{u_n}{n}$ and $\seq{v_n}{n}$ are linearly dependent, then so are $\seq{u_{2n}}{n}$
and $\seq{v_{2n}}{n}$. Assume thus that $\seq{u_n}{n}$ and $\seq{v_n}{n}$ are linearly independent, but $\seq{u_{2n}}{n} = \ell \seq{v_{2n}}{n}$ for some $\ell \in \R$.
Consider the solution $\seq[\infty]{U_n}{n=-1} := \seq{u_n}{n} - \ell \seq{v_n}{n}$ to
\eqref{eq: 100}. We have $U_{-1} = u_{-1} - \ell v_{-1}$ and $U_0 = 0$. Now $U_{-1} \neq 0$, as
otherwise $\seq{u_n}{n}$ would be proportional to $\seq{v_n}{n}$. Observe now that
$U_{1} = \tfrac{\gamma}{1+\alpha}U_{-1}$ and $U_2 = \tfrac{\beta}{\alpha+2} U_1 = \tfrac{\beta \gamma}{(1+\alpha)(2+\alpha)}U_{-1} \neq 0 = u_2-\ell v_2$.
This is a contradiction. We have established that $\seq{u_{2n}}{n}$ and $\seq{v_{2n}}{n}$
are necessarily linearly independent.

Assume now that $\seq{u_n}{n}$ is a minimal solution to \eqref{eq: 100} and let $\seq{v_n}{n}$ be a dominant solution. 
Then 
$0 = \lim_{n\to \infty} {u_n}/{v_n} = \lim_{n\to \infty} {u_{2n}}/{v_{2n}}.$
Since $\seq{u_{2n}}{n}$ and $\seq{v_{2n}}{n}$ are linearly independent solutions to \eqref{eq:modifiedRecurrence2}, $\seq{u_{2n}}{n}$ is necessarily minimal.

Conversely, assume that the linearly independent solutions 
$\seq{u_n}{n}$ and $\seq{v_n}{n}$ to \eqref{eq: 100} are such that $\seq{u_{2n}}{n}$ is a
minimal solution to \eqref{eq:modifiedRecurrence2} (such a solution exists by the previous paragraph, as we assume \eqref{eq: 100} to admit a minimal solution). Then, since $\seq{v_{2n}}{n}$ is linearly
independent to $\seq{u_{2n}}{n}$, we have $0 = \lim_{n\to\infty} {u_{2n}}/{v_{2n}}$. As the
recurrence relation \ref{eq: 100}  admits a minimal solution by assumption, it can be shown
that the limit $\lim_{n\to\infty}{u_n}/{v_n}$
exists (\cite[\S IV.1.5]{LW1992}), so $\lim_{n\to\infty} {u_n}/{v_n} = 0$. We have established that $\seq{u_n}{n}$ is
a minimal solution to \eqref{eq: 100}.
\end{proof}

\begin{corollary}\label{cor:minimalityReductionForms}
The decidability of \textsc{Minimality}$(j,k,\ell)$ with $j,k,\ell \in \{0,1\}$ reduces to deciding
\textsc{Minimality}$(0,1,0)$ and the Minimality Problem for solutions to recurrences of the
form \eqref{eq: lowDegreeForm111}, \eqref{eq: lowDegreeForm111RepeatedRoots},
\eqref{eq: lowDegreeForm110}, and \eqref{eq: lowDegreeForm101}.
\end{corollary}

Before delving into the proof of \autoref{thm:MinimalityRedToPeriodLike}, we establish some notation.
Consider the recurrences
\eqref{eq: lowDegreeForm111}--\eqref{eq: lowDegreeForm101}.
Dividing through by $g_3(n) = n+\alpha$, i.e., putting such recurrences into the form
\eqref{eq: ogformb}, we obtain Poincar\'{e} recurrences, since
$\lim_{n\to\infty}{g_2(n)}/{g_3(n)} = \beta_1$ and
$\lim_{n\to\infty}{g_1(n)}/{g_3(n)} = \gamma_1$. 
Let $\lambda$ and $\Lambda$ be the roots
of the associated characteristic polynomial such that $|\lambda| \leq |\Lambda|$. As
$\beta_1$ and $\gamma_1$ are not zero simultaneously in these recurrences, at least one of
the roots is non-zero. Hence $\Lambda \neq 0$.

Now recurrences of the form \eqref{eq: lowDegreeForm111RepeatedRoots} are a subset of
recurrences of the form \eqref{eq: lowDegreeForm111}. To avoid cluttering the text we shall
differentiate the two as follows: when referring to recurrences of the form
\eqref{eq: lowDegreeForm111} we always assume that $\beta_1^2 + 4\gamma_1 \neq 0$. Thus the characteristic roots are always distinct for relation \eqref{eq: lowDegreeForm111},
while \eqref{eq: lowDegreeForm111RepeatedRoots} has a single repeated characteristic root.

The next lemma gives necessary and sufficient conditions for the existence of minimal solutions.

%
\begin{lemma}\label{lem:whenHasMinimalSolution} \hfil
	\begin{enumerate}
	\item Recurrence relations associated to $\textsc{Minimality}(0,1,0)$ always admits a minimal solution.
%
	\item A recurrence of the form \eqref{eq: lowDegreeForm111}, \eqref{eq: lowDegreeForm110}, or
\eqref{eq: lowDegreeForm101} admits a minimal solution if and only if the associated
characteristic roots are real.
	\item The recurrence \eqref{eq: lowDegreeForm111RepeatedRoots} admits a minimal solution if and only if $\beta_0 - \alpha_0 - \gamma_0 \geq 0$.
	\end{enumerate}
\end{lemma}
\begin{proof}
We can determine whether a minimal solution exists using the criteria in \autoref{thm:contFracConvChar} that discusses the asymptotic properties of the function $r(n) = 1 + 4\frac{g_1(n)g_3(n-1)}{g_2(n)g_2(n-1)}$.

In Items 2 and 3, notice that the characteristic roots are real if and only if $\beta_1^2 + 4 \gamma_1 \geq 0$.
Further, the characteristic roots are distinct if and only if $\beta_1^2 + 4\gamma_1 \neq 0$.
	\begin{enumerate}
	\item  We have $r(n) = 1 + o(1)$. By \autoref{thm:contFracConvChar}, the recurrence
	always admits a minimal solution.
	\item  First, consider recurrence \eqref{eq: lowDegreeForm111}. Here we have
$r(n) = c_0 + o(1)$,
where $c_0 = 1+4\gamma_1/\beta_1^2$. 
Since we assume $\beta_1^2 + 4\gamma \neq 0$, we have that  $c_0 \neq 0$.
 By \autoref{thm:contFracConvChar}, the recurrence admits a minimal
solution if and only if $c_0 > 0$. 
But $c_0 > 0$ if and only if $\lambda$, $\Lambda \in \R$.
Second, consider recurrence \eqref{eq: lowDegreeForm110}.  In this case $r(n) = 1 + o(1)$.
 So the
recurrence always has a minimal solution and, in addition, we have $\lambda = 0$ and $\Lambda = \beta_1 \in \R$.
Finally, consider recurrence \eqref{eq: lowDegreeForm101}. 
In this case $r(n) = 1 + 4n^2/\beta_0^2 + o(n^2)$. 
The recurrence has a minimal solution by \autoref{thm:contFracConvChar} and the
characteristic roots are $\pm 1$.

	\item Consider recurrence \eqref{eq: lowDegreeForm111RepeatedRoots}. Here
$r(n) = {c_1}n^{-1} + {c_2}n^{-2} + o(n^{-2})$, where $c_1 = \beta_0 - \alpha - \gamma_0$,
and
$ c_2 = (\beta_0-\alpha)c_1 + (\alpha - \beta_0/2)(\alpha_0 - \beta_0/2 - 1)$.
There are two cases to consider. If $c_1 \neq 0$, then the recurrence admits a minimal solution if and only if $c_1 > 0$. 
Now assume that $c_1 = 0$, i.e., $\beta_0 = \alpha + \gamma_0$. 
Then
$c_2 = \frac{\alpha - \gamma_0}{2}(\frac{\alpha - \gamma_0}{2}-1)$. In this case the recurrence admits
a minimal solution if and only if $c_2 \geq -{1}/{4}$. 
This inequality is always true since \(x(x-1) + 1/4 = (x-1/2)^2 \ge 0\).
\qedhere
 \end{enumerate}
\end{proof}

In the following subsection, we show that \autoref{thm:MinimalityRedToPeriodLike} holds
for $\textsc{Minimality}(0,1,0)$. We then consider the other recurrences thereafter.

\subsection{\texorpdfstring{{Minimality}$(0,1,0)$}{\textsc{Minimality}(0,1,0)}}\label{subsec:Minimality010}

In this subsection we consider solutions to the recurrence
\begin{equation}\label{eq:010Recurrence}
u_{n} = (\beta_1 n + \beta_0)u_{n-1} + \gamma_0 u_{n-2}.
\end{equation}

For a minimal solution $\seq{u_n}{n}$ to this recurrence, we have
\begin{equation*}
-u_0/u_{-1} = \KF_{n=1}^{\infty} \mleft( \frac{\gamma_0}{\beta_1 n + \beta_0} \mright)
= \beta_0 \frac{
				\hypgeo{0}{1}(;\beta_0/\beta_1;\gamma_0 /\beta_1^2)
			}{
				\hypgeo{0}{1}(;\beta_0/\beta_1+1;\gamma_0 /\beta_1^2)
			}-\beta_0
\end{equation*}
(see \cite[\S VI.4.1]{LW1992}). Hence the Minimality Problem for the above recurrence reduces to checking the equality
\begin{equation}\label{eq:hypergeometricEquation}
(\beta_0u_{-1} - u_0) \hypgeo{0}{1} (;\beta_0/\beta_1+1;\gamma_0 /\beta_1^2) 
 = \beta_0 \hypgeo{0}{1}(;\beta_0/\beta_1;\gamma_0 /\beta_1^2).
\end{equation}
The \textit{Bessel functions of the first kind}, sometimes called \textit{cylinder functions}, \(J_s(z)\) are a family of functions that solve Bessel's differential equation \cite{abramowitz1972handbook, andrews1999special, cuyt2008handbook}.
For  \(z,s\in\C\) the function \(J_{s}(z)\) is defined by the hypergeometric series \cite[equation 9.1.69]{abramowitz1972handbook}
	\begin{equation*}
		J_s(z) := \sum_{k=0}^\infty \frac{(-1)^k}{k! \Gamma(s+k+1)} \mleft(\frac{z}{2} \mright)^{2k+s} = \frac{1}{\Gamma(s+1)} \mleft( \frac{z}{2} \mright)^{s} \hypgeo{0}{1}(; s+1; -z^2/4).
	\end{equation*}
We obtain the principal branch of \(J_{s}(z)\) by assigning \((z/2)^s\) its principal value.
When \(\Re(s)>-1/2\)  we have the following integral representation \cite[equation 4.7.5]{andrews1999special},
	\begin{equation*}
		J_s(z) = \frac{1}{\sqrt{\pi}\Gamma(s+1/2)} \mleft(\frac{z}{2}\mright)^s \int_{-1}^{1} \eu^{\iu zt} (1-t^2)^{s-1/2} \, dt.
	\end{equation*}
%
%
%
Hence for \(\Re(s)>-1/2\),  we have the following integral representation
\begin{equation*}
\hypgeo{0}{1}(;s+1;z) =
		\frac{\Gamma(s+1)}{\sqrt{\pi}\Gamma(s+1/2)}
			\int_{-1}^{1} \eu^{-2 \sqrt{z} t} (1-t^2)^{s-1/2} \, dt.
\end{equation*}
Let us return to minimal solutions of the aforementioned recurrence relation.
By substitution into \eqref{eq:hypergeometricEquation} and linearity of the integral, 
 we see that
$\textsc{Minimality}(0,1,0)$ reduces to checking the equality
	\begin{equation} \label{eq:integralEquation010}
	\int_{-1}^1   \eu^{-2\sqrt{\gamma_0} t/\beta_1} (1-t^2)^{\beta_0/\beta_1 - 3/2}  \mleft( \frac{2(\beta_0 u_{-1} - u_0)}{2\beta_0-1}(1-t^2) -1 \mright)\, dt = 0.
	\end{equation}
To ensure that the integral converges absolutely note that we can shift the recurrence so that ${\beta_0}/{\beta_1} > 3/2$.
The integral on the left-hand side is an exponential period, and thus we have proved
\autoref{thm:MinimalityRedToPeriodLike} in the case of $\textsc{Minimality}(0,1,0)$.

\subsection{Generating function analysis}
For the remainder of this section, we only consider recurrences
\eqref{eq: lowDegreeForm111}--\eqref{eq: lowDegreeForm101}. By \autoref{lem:whenHasMinimalSolution}, we thus
assume that the associated characteristic roots are real. We may now choose $\Lambda > 0$:
indeed, $\beta_1 > 0$ in the first three recurrences implies that the dominant root is
positive. In the fourth recurrence, the roots are $\pm 1$, and we are free to choose
$\Lambda = 1$. (The assumption of $\beta_0 > 0$ in \eqref{eq: lowDegreeForm101} is used in the
sequel. We will explicitly recall this fact when needed, but, for now, this is not important.)

Let $\seq{u_n}{n \geq -1}$ be a non-trivial
solution to one of the recurrences of the form 
\eqref{eq: lowDegreeForm111}--\eqref{eq: lowDegreeForm101}. 
We associate to $\seq{u_n}{n}$ the
 generating series
\begin{equation*}
\mathcal{F}(x) = \sum_{n=0}^{\infty} u_{n-1} x^{n+\alpha}.
\end{equation*}
We consider analytic properties of the generating function defined by the above generating series. We first observe that the series has a positive radius of convergence. Recall that
$\Lambda > 0$ in this analysis.

\begin{lemma}\label{lem:radiusOfConvergence}
Let $\seq{u_n}{n}$ be a non-trivial solution to one of aforementioned recurrences. If
$\seq{u_n}{n}$ is a dominant (resp., minimal) solution, then the series $\mathcal{F}(x)$
has radius of convergence ${1}/{\Lambda}$
(resp., ${1}/{|\lambda|}$, which we understand as $\infty$ if $\lambda = 0$.)
\end{lemma}
\begin{proof}
This follows from \autoref{thm: perron} by the Cauchy--Hadamard theorem.
\end{proof}

Now the generating series $\mathcal{F}$ defines a continuous and differentiable function
in the interval $[0,{1}/{|\Lambda|})$ (regardless of whether $\seq{u_n}{n}$ is dominant or
not). Further, $\mathcal{F}$ is analytic in the interval
$(0,1/|\Lambda|)$. It can be shown (see \autoref{app:ode}) that for a given solution $\seq{u_n}{n}$, the generating function satisfies the differential equation
\begin{equation*}
\mathcal{F}'(x) + s(x) \mathcal{F}(x) = t(x),
  \end{equation*}
where
\begin{align}\label{eq:sx}
s(x) &= \frac{( \gamma_0 + 2 \gamma_1 - \alpha \gamma_1)x
+ \beta_0 + \beta_1 - \alpha\beta_1}{\gamma_1x^2 + \beta_1 x - 1 } \quad \text{and} \\
t(x) &= \frac{(\beta_0+\beta_1)u_{-1} - (\alpha + 1)u_0 - \alpha u_{-1}x^{-1} }	{\gamma_1 x^2 + \beta_1 x - 1}x^{\alpha}. \nonumber
\end{align}
Note that both $s$ and $t$ are integrable over a neighbourhood of $0$
under our assumption that $\alpha > 1$.
Standard methods then yield a solution to the differential equation. Namely,
by noting that $\mathcal{F}(0)=0$ as $\alpha > 1$, we have the
following solution in the interval $[0,1/\Lambda)$ as follows:
\begin{equation}\label{eq:DEGeneralSolution}
\mathcal{F}(x) = \frac{1}{\mathcal{I}(x)} \int_{0}^{x} \mathcal{I}(y) t(y)\, dy,
\end{equation}
where $\mathcal{I}(x) $ is the integrating factor $\exp(\int_0^x s(y)\, dy)$.

Notice here that the denominator $\gamma_1 x^2 + \beta_1 x - 1$ of $t(x)$ and $s(x)$ has roots $1/\Lambda$ and $1/\lambda$ (if $\lambda \neq 0$).

\begin{lemma}\label{lem:vanishingIntegral}
There is at most one non-trivial choice of $u_{-1}$ and $u_0$, up to scaling, for which
$\int_{0}^{1/\Lambda} \mathcal{I}(y)t(y)\,dy$ vanishes.
\end{lemma}
\begin{proof}
Observe that the integral is of the form
\begin{equation*}
\int_{0}^{1/\Lambda} \frac{ \mathcal{I}(y) }{\gamma_1 y^2 + \beta_1 y - 1}y^{\alpha-1} (A y + B)\, dy
\end{equation*}
where only $A$ and $B$ depend on $u_{-1}$ and $u_0$. 
We note $\mathcal{I}(y)t(y)/(Ay+B)$ has constant sign in the domain \((0,1/\Lambda)\).
This is easily seen by analysing the signs of the numerator and denominator.
For the numerator, \(\mathcal{I}(y)y^{\alpha-1}\) and for the quadratic in the denominator, we note the domain 
is either totally contained
in $[1/\lambda,1/\Lambda]$ if $\lambda < 0$ or is disjoint from $[1/\Lambda, 1/\lambda)$ (resp., $[1/\Lambda, \infty)$) if $\lambda > 0$ (resp., $\lambda = 0$).
%

Assume that the integral vanishes for some choice of $u_{-1}$ and $u_0$. We first show that
$u_{-1} \neq 0$. Indeed, if $u_{-1} = 0$, then the integral takes the form
\begin{equation*}
A \int_{0}^{1/\Lambda} \frac{ \mathcal{I}(y) }{\gamma_1 y^2 + \beta_1 y - 1}y^{\alpha}\, dy.
\end{equation*}

The integrand has constant sign and does not vanish on \([0,1/\Lambda)\) (unless $A=0$, which occurs precisely when $u_0 = u_{-1} = 0$).
We deduce that the integral does not vanish.

Assume now that the integral vanishes for two distinct pairs \((u_{-1}, u_0)\) and \((u_{-1}', u_0')\). As $u_{-1}, u_{-1}' \neq 0$ we can, without loss of generality, assume the pairs are of the form \((1, u_0)\) and \((1, u_0')\) with \(u_0\neq u_0'\).
Substitution of these initial values into the integral gives the following
\begin{equation*}
\int_{0}^{1/\Lambda} \frac{ \mathcal{I}(y) }{\gamma_1 y^2 + \beta_1 y - 1}y^{\alpha - 1}
(Ay + B)\, dy
= 0 =
\int_{0}^{1/\Lambda} \frac{ \mathcal{I}(y) }{\gamma_1 y^2 + \beta_1 y - 1}y^{\alpha - 1}(A' y + B)\, dy.
\end{equation*}
By linearity, we conclude that the integral vanishes with the choice \((0, u_1-u_1')\).
This contradicts our earlier observation and concludes the proof.
\end{proof}

The following lemma makes explicit the integrands $\mathcal{I}(y)t(y)$ of the recurrences
\eqref{eq: lowDegreeForm111}--\eqref{eq: lowDegreeForm101}.

\begin{lemma}\label{lem:DifferentGFForms}
The integrating factor $\mathcal{I}(x)$ in \eqref{eq:DEGeneralSolution}, is as follows.
\begin{enumerate}
\item For recurrences \eqref{eq: lowDegreeForm111} and \eqref{eq: lowDegreeForm101} we have
\begin{equation*}
\mathcal{I}(x) = ({1}/{\Lambda}-x)^{\nu} |x-{1}/{\lambda}|^{\nu'}
= \frac{1}{\Lambda^{\nu} |\lambda|^{\nu'}}
(1-\Lambda x)^{\nu} |1-\lambda x|^{\nu'},
\end{equation*}
where $\nu = \mathcal{A}(\lambda)$, $\nu' = -\mathcal{A}(\Lambda)$, and
\begin{equation*}\label{eq:mathcalA}
   \mathcal{A}(x) = \frac{ (\alpha - \gamma_0/\gamma_1 - 2)x + (\beta_0+\beta_1 -\alpha \beta_1)}{\Lambda-\lambda}.
\end{equation*}

\item For recurrence \eqref{eq: lowDegreeForm111RepeatedRoots} we have
\begin{equation*}
\mathcal{I}(x) = \exp \mleft(\frac{\beta_0 - \alpha -\gamma_0}{x-1}\mright) (1-x)^{2 + \gamma_0 -\alpha}.
\end{equation*}

\item For recurrence \eqref{eq: lowDegreeForm110} we have 
\begin{equation*}
\mathcal{I}(x) = \eu^{\nu' x} (\beta_1^{-1} - x)^{\nu}
 = \beta_1^{-\nu} \eu^{\nu' x} (1 - \beta_1 x)^{\nu}
\end{equation*}
where $\nu' = \gamma_0 / \beta_1$ and $\nu =  {\beta_0}/{\beta_1} - \alpha + 1 + \gamma_0/\beta_1^2$.
\end{enumerate}
\end{lemma}
We have deliberately chosen to share notation between the different items in the above
lemma (especially the parameter $\nu$): in the sequel, we shall treat several of the cases simultaneously.
\begin{proof}
The first two claims follow from a straightforward integration of partial fractions. 
One only
notes that the roots of $\gamma_1 x^2 + \beta_1 x - 1$ are $1/\lambda$ and $1/\Lambda$.
Further, in the first claim we have
$s(x) = \frac{\mathcal{A}(\lambda)}{x-1/\Lambda} - \frac{\mathcal{A}(\Lambda)}{x-1/\lambda}$.  
In the second claim we have 
	\begin{equation*}
		s(x) = \frac{2 + \gamma_0 - \alpha_0}{x-1} + \frac{\alpha_0 + \gamma_0 - \beta_0}{(x-1)^2}.
	\end{equation*}
 We remark that in this case we substitute $\beta_1 = 2$, $\gamma_1 = -1$, and $\gamma_0 \mapsto -\gamma_0$ in \eqref{eq:sx}.

In the last case, we substitute $\gamma_1 =0$, so 
	\begin{equation*}
		s(x) = \frac{\gamma_0}{\beta_1} + \frac{\beta_0/\beta_1 + 1 - \alpha + \gamma_0/\beta_1^2}{x-1/\beta_1}. \qedhere
	\end{equation*}
\end{proof}

Our approach to deciding minimality hinges on identifying when $\mathcal{F}(x)$ and all of
its (left) derivatives exist at $1/\Lambda$. Indeed, when $\seq{u_n}{n}$ is a minimal
solution to \eqref{eq: lowDegreeForm111}, then $\mathcal{F}(x)$ is known to be analytic in a neighbourhood of \(1/\Lambda\).
As we shall shortly see, this connection holds for all the cases at hand. Let us first discuss the Minimality Problem for instances of
\eqref{eq: lowDegreeForm111RepeatedRoots}.

\subsection{Minimality for recurrence \texorpdfstring{\eqref{eq: lowDegreeForm111RepeatedRoots}}{111RepeatedRoots}}
Recall that recurrence \eqref{eq: lowDegreeForm111RepeatedRoots} admits a minimal
solution if and only if $\beta_0 \geq \alpha + \gamma_0$ by
\autoref{lem:whenHasMinimalSolution}.
We shall consider the cases $\beta_0 > \gamma_0 + \alpha$ and $\beta_0 = \gamma_0 + \alpha$
separately.

Recall from \autoref{ex:KoomanAsymptotics} that in the case $\beta_0 > \gamma_0 + \alpha$, recurrence
\eqref{eq: lowDegreeForm111RepeatedRoots} admits two linearly independent solutions \(\seq{u_n^{(1)}}{n}\) and \(\seq{u_n^{(2)}}{n}\) such that
\begin{align*}
u_n^{(1)} &\sim  n^{(1-2\alpha +2\gamma_0)/4} \exp(2\sqrt{(\beta_0 - \alpha - \gamma_0)n}) \quad \text{and} \\
u_n^{(2)} &\sim  n^{(1-2\alpha +2\gamma_0)/4} \exp(-2\sqrt{(\beta_0 - \alpha - \gamma_0)n}).
\end{align*}
Notice that $\seq{u_n^{(2)}}{n}$ is a minimal solution and that $\seq{u_n^{(1)}}{n}$ is
dominant. 
From the asymptotics above, it is evident that
$\sum_{n=0}^{\infty} u_{n-1}^{(1)} = \infty$ and $\sum_{n=0}^{\infty} u_{n-1}^{(2)}$ is finite. 
Hence, by Abel's theorem, 
we have that $\lim_{x \to 1-} \mathcal{F}(x) = \infty$ (resp., $\lim_{x \to 1-} \mathcal{F}(x)$ is finite) for the generating function $\mathcal{F}$.
In particular, we have proved the following lemma.
\begin{lemma} \label{lem:minimalcriterionF}
A non-trivial solution $\seq{u_n}{n}$ to \eqref{eq: lowDegreeForm111RepeatedRoots} with $\beta_0 > \gamma_0 + \alpha$ is minimal if and only
if the left limit of corresponding generating $\lim_{x\to 1-}\mathcal{F}(x)$ exists and is finite.
\end{lemma}

\begin{proposition}\label{prop:minimalityRepeatedRoots}
Let $\seq{u_n}{n}$ be a non-trivial solution to \eqref{eq: lowDegreeForm111RepeatedRoots} where
$\beta_0 > \alpha + \gamma_0$. Then $\seq{u_n}{n}$ is minimal if and only if
\begin{equation*}
\int_{0}^{1} \exp \left( \frac{\beta_0 - \alpha - \gamma_0}{y-1} \right)
	(1-y)^{\gamma_0 - \alpha} y^{\alpha-1}(Ay - B)\, dy = 0,
\end{equation*}
where $A = (2+\beta_0)u_{-1} - (\alpha+1)u_0$ and $B = \alpha u_{-1}$.
\end{proposition}
\begin{proof}
Recall that a non-trivial solution to \eqref{eq: lowDegreeForm111RepeatedRoots} defines the generating
function $\mathcal{F}(x)$ as in
\autoref{lem:DifferentGFForms}(2). The integral in this claim is the integral
$\int_0^1 \mathcal{I}(y)t(y)\, dy$. 
The integral is absolutely converging as
$\lim_{x\to 1-}\exp (\frac{\beta_0 - \alpha - \gamma_0}{x-1}) = 0$ 
since $\beta_0 > \alpha + \gamma_0$.
(In particular, the factor
$(1-y)^{\gamma_0 - \alpha}$ does not affect the convergence.)

Suppose that \(\seq{u_n}{n}\) is a minimal sequence and let $\zeta$ be the value of the associated integral.  
Assume, for a
contradiction, that $\zeta \neq 0$. 
Then, by definition, $\lim_{x \to 1-} \mathcal{F}(x) = \sign(\zeta) \infty$, as $\mathcal{I}(x)^{-1}$ has a singularity at $x=1$. 
This contradicts the criterion in \autoref{lem:minimalcriterionF}.
Thus we conclude that the integral $\int_{0}^{1}\mathcal{I}(y)t(y)\,dy$ associated to a minimal solution \(\seq{u_n}{n}\) vanishes.

The converse argument follows from \autoref{lem:vanishingIntegral} and this concludes the proof.
%
%
%
\end{proof}

We then consider the case of $\beta_0 = \alpha + \gamma_0$. The Minimality Problem for this case
is decidable as evidenced by following proposition.
\begin{proposition}\label{prop:111repeatedRoot}
Let $\seq{u_n}{n}$ be a non-trivial solution to \eqref{eq: lowDegreeForm111RepeatedRoots} with
$\beta_0 = \alpha + \gamma_0$.
If \(\alpha\leq\gamma_0+1\) then $\seq{u_n}{n}$ is minimal if and only if $u_0/u_{-1} = 1$.
If \(\alpha > \gamma_0 +1\) then $\seq{u_n}{n}$ is minimal if and only if
$u_0/u_{-1} = \frac{\gamma_0 + 1}{\alpha}$.
\end{proposition}
\begin{proof}
We are in the setting of \autoref{rem:All1Sol} with $q_n = \frac{n+\gamma_0}{n+\alpha}$.
Hence $\seq{1}{n}$  and $\seq{v_n}{n}$ defined by $v_{-1} = 0$,
$v_0 = 1$ such that
\begin{equation*}
v_n = \sum_{k=0}^n \prod_{m=1}^k \frac{m + \gamma_0}{m+\alpha} = \sum_{k=0}^n \frac{(\gamma_0 + 1)_k}{(\alpha + 1)_k}
\end{equation*}
are linearly independent solution sequences.
By a straightforward application of Stirling's approximation, $(\gamma_0 +1)_n/(\alpha + 1)_n \sim n^{\gamma_0 - \alpha}$ as \(n\to\infty\). 
Hence if $\gamma_0 - \alpha \geq -1$ the series diverges
(by comparison to the harmonic series) from which we deduce that $\seq{1}{n}$ is the minimal solution. If $\gamma_0 - \alpha < -1$,
then $\lim_{n\to \infty} v_n$ converges to the value
\begin{equation*}
\sum_{k=0}^{\infty} \frac{(\gamma_0 + 1)_k (1)_k}{(\alpha + 1)_k}\frac{1}{k!}
= \hypgeo{2}{1}(\gamma_0+1,1; \alpha+1;1)
= \frac{\Gamma(\alpha + 1)\Gamma(\alpha - \gamma_0  -1)}{\Gamma(\alpha - \gamma_0)\Gamma(\alpha)}
= \frac{\alpha}{\alpha - \gamma_0 - 1}.
 \end{equation*}
 In the second equality we again use Theorem~2.2.2 from \cite{andrews1999special}.
 It follows that $\seq{u_n}{n} = \frac{\alpha}{\alpha - \gamma_0 - 1}\seq{1}{n} - \seq{v_n}{n}$ is a minimal solution, and we may compute $u_0/u_{-1} = \frac{\gamma_0 + 1}{\alpha}$.
\end{proof}

Notice that the integral in \autoref{prop:minimalityRepeatedRoots} above is an exponential
period. 
As the Minimality Problem in the case $\beta_0 = \alpha + \gamma_0$ is a decidable
problem, we conclude that \autoref{thm:MinimalityRedToPeriodLike} holds for recurrences of the
form \eqref{eq: lowDegreeForm111RepeatedRoots}.

\subsection{Minimality for recurrences \texorpdfstring{\eqref{eq: lowDegreeForm111},
\eqref{eq: lowDegreeForm110}, and \eqref{eq: lowDegreeForm101}}{111, 110, and 101}}
\label{subsec:Minimality111_110_101}
Consider recurrences \eqref{eq: lowDegreeForm111}, \eqref{eq: lowDegreeForm110}, and
\eqref{eq: lowDegreeForm101}. Recall that in these instances, the characteristic roots are
distinct (when omitting the subcase \eqref{eq: lowDegreeForm111RepeatedRoots}, which was handled above).
Define $f(x) = \mathcal{I}(x)t(x)$, i.e., $f(x)$ is the integrand of \eqref{eq:DEGeneralSolution}, and $\mathcal{I}$ is as in
\autoref{lem:DifferentGFForms}(1~\&~3). Recall that $\Lambda = \beta_1$ in
\eqref{eq: lowDegreeForm110} and $(1-\Lambda x)$ is a factor of the denominator of $t(x)$.
For \(x\in [0,1/\Lambda)\) sufficiently close to \(1/\Lambda\),
\begin{equation}\label{eq:integrandAsymptotics}
f(x) = (1 - \Lambda x)^{\nu-1} \sum_{n = 0}^\infty c_n (1-\Lambda x)^n = \sum_{n + \nu \leq 0} c_n (1-\Lambda x)^{n+ \nu - 1} +
\mathcal{O}\big( (1-\Lambda x)^{n_0 + \nu - 1} \big) 
\end{equation}
where $\nu$ is given in \autoref{lem:DifferentGFForms} and
$n_0 \in \N_0$ is the least integer such that $n_0 + \nu > 0$.
In the sequel we use the notation \(H(x) = \sum_{n = 0}^\infty c_n (1-\Lambda x)^n\) for brevity. 
%
\begin{claim}\label{clm:integralAsymptotics}
For $x \in [0,{1}/{\Lambda})$ sufficiently close to $1/\Lambda$,
\begin{equation}\label{eq:intfx}
\int_{0}^{x} f(y)\, dy = \sum_{n < -\nu}
			\frac{-c_n/\Lambda }{n+ \nu}(1-\Lambda x)^{n + \nu}
			+ C_0 + C_1 \log(1-\Lambda x) + \mathcal{O}((1-\Lambda x)^{n_0 + \nu}),
\end{equation}
where $C_1 = -c_{- \nu}/\Lambda$ if $-\nu \in \N_0$ and $C_1 = 0$ otherwise, and 
\begin{equation}\label{eq:C0}
C_0 = \sum_{n < -\nu} \frac{c_n/\Lambda}{n+ \nu}
 + \int_{0}^{1/\Lambda} f(y) - \sum_{n \leq -\nu } c_n(1-\Lambda y)^{n+\nu - 1}\, dy.
\end{equation}
\end{claim}
\begin{proof}
This can be seen by integrating the series termwise (cf.\, \cite[Theorem~VI.9(ii)]{flajolet2009analytic}.) See \autoref{subapp:IntegralAsymptotics} for a proof.
\end{proof}
We are interested in the behaviour of
$\mathcal{F}(x) = \mathcal{I}(x)^{-1} \int_{0}^{x} f(y)dy$ as $x \to 1/\Lambda-$.
By inspecting \autoref{lem:DifferentGFForms}(1~\&~3) we notice that in both cases
$\mathcal{I}(x)^{-1} = (1-\Lambda x)^{-\nu}$ multiplied by a function that
is analytic at $1/\Lambda$.
Multiply \eqref{eq:intfx} through by $\mathcal{I}^{-1}(x)$. Then, modulo the addition of
analytic terms that vanish at \(x=1/\Lambda\),
\(\mathcal{F}(x)\) is given by the product of an analytic function that does not vanish at $1/\Lambda$ and
\begin{equation} \label{eq:solutionSeries}
\sum_{n  < -\nu}\frac{-c_n / \Lambda}{n + \nu}(1-\Lambda x)^n
+ (C_0 + C_1 \log(1-\Lambda x) ) (1-\Lambda x)^{-\nu }. 
\end{equation}

The next lemma identifies when $\lim_{x\to 1/\Lambda-}\mathcal{F}^{(\ell)}(x)$ exists for each $\ell \geq 0$.

\begin{lemma}\label{lem:FallDerivativesExistCharacterisation}

Consider the parameter $\nu$ in \eqref{eq:integrandAsymptotics}.
\begin{enumerate}
\item Suppose that $-\nu \notin \N_0$.  Then the lower limits of $\mathcal{F}$ and its derivatives are finite at $1/\Lambda$ if and only if \(C_0=0\); that is,
\begin{equation}\label{eq:IntegralEquation}
\int_{0}^{1/\Lambda} f(y) - \sum_{n  \leq -\nu } c_n(1-\Lambda y)^{n+\nu - 1} \, dy = \sum_{n < -\nu} \frac{-c_n/\Lambda}{n+ \nu}.
\end{equation}
\item Suppose that $-\nu \in \N_0$. Then the lower limits of $\mathcal{F}$ and its derivatives are finite at $1/\Lambda$ if and only if
$C_1 = 0$.
\end{enumerate}
\end{lemma}
\begin{proof}
From \eqref{eq:solutionSeries} we have that the lower limits of
$\mathcal{F}$ and its derivatives are finite at $1/\Lambda$ if and only if this is so for the function \(C_0 (1-\Lambda x)^{-\nu} + C_1 \log(1-\Lambda x)(1-\Lambda x)^{-\nu}\).
\begin{enumerate}
\item Assume that $-\nu \notin \N_0$. Then $C_1 = 0$ by definition. Hence  the lower limits of $\mathcal{F}(x)$ and its left derivatives are finite
at $1/\Lambda$ if and only if $C_0 = 0$ and so we have the desired result.
\item 
Assume that $-\nu \in \N_0$. Then the first term in the above function is analytic at $1/\Lambda$. 
Hence the function $\mathcal{F}(x)$ and all its left derivatives exist at $1/\Lambda$ if and only if $C_1 = 0$.
%
%
\qedhere
\end{enumerate}
\end{proof}

%
%
Recall from \autoref{ex:KoomanAsymptotics} that recurrence \eqref{eq: lowDegreeForm101} admits two linearly independent solutions \(\seq{u_n^{(1)}}{n}\) and \(\seq{u_n^{(2)}}{n}\) such that 
\begin{equation*}
u_n^{(1)} \sim
	n^{\frac{1}{2}(\beta_0 + \gamma_0 - \alpha)}, \quad \text{and} \quad
u_n^{(2)} \sim \mleft(-1 \mright)^n
	n^{\frac{1}{2}(-\beta_0 + \gamma_0 - \alpha)}.
\end{equation*}
Since \(\beta_0>0\), \(\seq{u_n^{(1)}}{n}\) is a dominant solution.
For \(\ell\) sufficiently large it is immediate that
$\sum_{n=0}^{\infty} n(n-1)\cdots (n-\ell) u_{n-1}^{(1)} = \infty$.
By Abel's theorem, $\lim_{x \to 1-}\mathcal{F}^{(\ell)}(x) = \infty$ for the corresponding generating
function $\mathcal{F}$.
We thus have the following corollary.

\begin{corollary}\label{cor:dominantSolutionToInfinity}
For any dominant solution to \eqref{eq: lowDegreeForm101}, the corresponding
generating function has a derivative for which $\lim_{x\to 1-}\mathcal{F}^{(\ell)}(x) = \pm \infty$.
\end{corollary}

We will now show for each recurrence relation \eqref{eq: lowDegreeForm111},
\eqref{eq: lowDegreeForm110}, and \eqref{eq: lowDegreeForm101}, there exists a choice of
$u_{-1}$ and $u_0$ such that $\mathcal{F}$ and all its left derivatives exist at $1/\Lambda$.

\begin{lemma}
For each of the recurrences \eqref{eq: lowDegreeForm111}, \eqref{eq: lowDegreeForm110}, and
\eqref{eq: lowDegreeForm101}, there is a choice of $u_0$ and $u_{-1}$ such that the left limits of the corresponding
generating function $\mathcal{F}$ and all its derivatives exist at $1/\Lambda$.
\end{lemma}
\begin{proof}
For recurrences \eqref{eq: lowDegreeForm111} and \eqref{eq: lowDegreeForm110} the result follows from \autoref{lem:radiusOfConvergence} 
as a minimal solution necessarily defines such $u_{-1}$ and $u_0$. 
We may thus focus on recurrence
\eqref{eq: lowDegreeForm101}.
Let us put $u_{-1} = 1$. By \autoref{lem:DifferentGFForms}(3) we have
\begin{equation*}
f(y) = (1-y)^{\nu - 1} (1+y)^{\nu'-1} y^{\alpha-1}(Ay - \alpha)
\end{equation*}
where $\nu = 1 + \frac{ \gamma_0 - \alpha + \beta_0}{2}$,
$\nu' = \frac{\alpha - \gamma_0 +\beta_0}{2} - 1= -\nu + \beta_0$,
$A = \beta_0 - (\alpha + 1)u_0$.
In terms of \eqref{eq:integrandAsymptotics}, we have
$\sum_{n=0}^{\infty} c_n (1-\Lambda x)^n = (1+y)^{\nu'-1} y^{\alpha-1}(Ay - \alpha)$.

In light of \autoref{lem:FallDerivativesExistCharacterisation}, we consider two cases. (Recall
here that $\Lambda = 1$ and $\lambda = -1$.)
\begin{enumerate}
\item
Assume first that $-\nu \notin \N_0$. Then, by \autoref{lem:FallDerivativesExistCharacterisation}(1),
we need to establish a choice of $u_0$ for which \eqref{eq:IntegralEquation} holds.
Assume first that $\nu > 0$. Then equation \eqref{eq:IntegralEquation} is simply
$\int_{0}^{1} f(y)\,dy = 0$, which is equivalent to
\begin{equation*}
A\int_{0}^{1} (1-y)^{\nu - 1} (1+y)^{\nu'-1}y^{\alpha}\,dy = 
\alpha\int_{0}^{1} (1-y)^{\nu-1} (1+y)^{\nu'-1}y^{\alpha-1}\,dy.
\end{equation*}
Clearly we can choose $u_0$ for this equation to hold as the integral on the left does not vanish (the integrand is strictly positive at each point in the domain of integration).

Assume that $\nu < 0$ (and recall that $-\nu \notin \N_0$).
Consider the parameters in \eqref{eq:IntegralEquation}.
We write
$(1+y)^{\nu'-1}y^{\alpha-1} = \sum_{n=0}^{\infty} d_n (1-y)^n$ when $y$ is close to $1$. Then
$\sum_{n=0}^{\infty} c_{n}(1-y)^n = (A y - \alpha)\sum_{n=0}^{\infty}d_{n}(1-y)^n$,
so that $c_n = A(d_n - d_{n-1}) - \alpha d_{n}$ (with the convention $d_{-1} = 0$).
Recall that $n_0$ is the least element of $\N_0$ such that $n_0 + \nu > 0$. We have
\begin{equation*}
\sum_{n \leq - \nu}c_n(1-y)^n
=    (Ay - \alpha) \sum_{n \leq - \nu} d_n(1-y)^n + A d_{n_0 - 1}(1-y)^{n_0}.
\end{equation*}
Now recall that
$\Lambda = 1$ and $\lambda = -1$ in \eqref{eq: lowDegreeForm101}.
Then the equality in \eqref{eq:IntegralEquation} holds if and only if the following expression is equal to zero
\begin{equation*}
\int_{0}^{1} (A y - \alpha)\biggl( K(y) - \sum_{n \leq -\nu} \!\!\! d_n (1-y)^{n  + \nu - 1}\biggr) \, dy
-A\frac{d_{n_0-1}}{n_0 + \nu} + \sum_{n  < -\nu}\frac{A(d_n - d_{n-1}) - \alpha d_{n}}{n+ \nu}.
\end{equation*}
Here $K(y) := (1-y)^{\nu - 1} (1+y)^{\nu'-1} y^{\alpha-1}$.
We shall henceforth call the above expression \(J(A)\).
We now show that \(J(A)\) is continuous in $A$ (and so is continuous in $u_0$).
%
\begin{claim}
The function $J$ is continuous. Moreover, it is differentiable for $A$ in any open, bounded interval $(a,b)$.
\end{claim}
\begin{proof}
To show that \(J\) is differentiable, it suffices to show that one can \textit{pass the differentiation under the integral sign} \cite[\S20.4]{priestley1997integration} so that
	\begin{equation*}
	\frac{d}{dA} \int_{0}^{1} (A y - \alpha)\biggl( K(y) - \sum_{n \leq -\nu} \!\!\! d_n (1-y)^{n  + \nu - 1}\biggr) \, dy = \int_0^1 \frac{\partial}{\partial A} g(y,A) \, dy.
	\end{equation*}
This process is permitted because
$y\mapsto g(y,A)$ is integrable by construction;
the derivative $\partial g(y,A)/\partial A = yK(y) - y \sum_{n\leq -\nu} d_n(1-y)^{n+\nu - 1}$ exists for each \(y\in(0,1)\); and
\(|\partial g(y,A)/\partial A|\) is an integrable function independent of \(A\).
%
\end{proof}

Now, by Leibniz's rule we have
\begin{equation*}
J'(A) = \int_{0}^{1} y K(y) - y\sum_{n  \leq - \nu} \!\!\! d_n (1-y)^{n + \nu - 1} \, dy - \frac{d_{n_0-1}}{n_0 + \nu}
+ \sum_{n < -\nu}\frac{d_n - d_{n-1}}{n+\nu}.
\end{equation*}
Hence $J$ is either constant or linear in $A$ depending on whether $J'(A) = 0$.

\begin{claim}
We claim that $J'(A) \neq 0$.
\end{claim}
\begin{proof}
Recall that a dominant solution to \eqref{eq: lowDegreeForm101} defines a generating function
for which $\lim_{x \to 1-} \mathcal{F}^{(\ell)}(x) = \pm \infty$ for sufficiently large $\ell$. 
From \eqref{eq:solutionSeries}, we notice that the sign of the above limit is
determined by the sign of $C_0$ (since the analytic term omitted from the expression does not
vanish at $1/\Lambda$). It thus suffices to exhibit two choices of $u_0$ (with $u_{-1} = 1$)
for which both the signs in the limit are realised.

%
Now the solution $\seq{v_n}{n}$ defined by $v_{-1} = 0$, $v_0 = 1$, to
\eqref{eq: lowDegreeForm101} is dominant. Indeed, \autoref{prop:pos} together with the
discussion preceding it imply that the sequence $(-1)^n v_n$ is dominant (it satisfies a
recurrence with signature $(-,+)$), and dominance is inherited by $\seq{v_n}{n}$. Let
$\seq{v_n'}{n}$ be a
minimal solution to \eqref{eq: lowDegreeForm101}. We may assume that $v'_{-1} = 1$, as it is linearly independent to $\seq{v_n}{n}$. Then the solutions given by $\pm \seq{v_n}{n} + \seq{v_n'}{n}$
%
define generating functions that each diverge to \(\pm \infty\) (with opposite signs) and \(\pm v_0 + v'_0 =1\) as required.
\end{proof}
We deduce that $J(A)$ is a degree one polynomial in $A$. In particular, there is a choice of $A$ such that $J(A) = 0$. This concludes the proof of the first case.

\item
Assume second that $-\nu \in \N_0$. Then, by \autoref{lem:FallDerivativesExistCharacterisation}(2),
we need to exhibit a choice of $u_0$ for which $C_1 = 0$. This is equivalent to $c_{-\nu} = 0$. Recall that
$\sum_{n=0}^{\infty} c_n(1-y)^n = (A y - \alpha) (1+y)^{\nu'-1}y^{\alpha-1} =: H(y)$. So,
$c_{-\nu} = 0$ if and only if $H^{(-\nu)}(1) = 0$. We further introduce
$K(y) = (1+y)^{\nu'-1}y^{\alpha-1}$ (so $H(y) = K(y)(Ay -\alpha)$) and $\ell = -\nu$.
Now if $\ell = 0$, then
$u_0 = \frac{\beta_0 - \alpha}{\alpha + 1}$ forces $H(1) = 0$. For $\ell \in\N$ we have
\begin{equation*}
H^{(\ell)}(x) = \sum_{k=0}^{\ell}\binom{\ell}{k} K^{(\ell-k)}(x) \frac{d^k}{dx^k}(Ax + B) =
K^{(\ell)}(x) (A x - \alpha) + \ell K^{(\ell - 1)}(x)A
\end{equation*}
and so $H^{(\ell)}(1) = A(K^{(\ell)}(1) + \ell K^{(\ell-1)}(1)) - \alpha K^{(\ell)}(1)$.
As long as $K^{\ell}(1) + \ell K^{\ell - 1}(1) \neq 0$, we can choose $u_0$ in a suitable way
to force $H^{(\ell)}(1) = 0$. Next we show that $K^{(\ell)}(1) + \ell K^{\ell}(1) \neq 0$ to conclude the proof.
Recall that \(K(y) = (1+y)^{\nu'-1}y^{\alpha - 1}\)
we have
\begin{equation*}
K^{(\ell)}(y) = \sum_{k=0}^{\ell} \binom{\ell}{k} \frac{d^{k}}{dy^k}(1+y)^{\ell + \beta_0 - 1} \frac{d^{\ell - k}}{dy^{\ell - k}}y^{\alpha-1}.
\end{equation*}
In our current working $\beta_0,\gamma_0 > 0$ and so we have the following inequalities.
First,
 \(\nu' = -\nu + \beta_0 = \ell + \beta_0 > \ell\). 
Second,
\(\alpha = 2\nu' + \gamma_0 + 2 - \beta_0 = 2\ell + \gamma_0 + \beta_0 + 2 > \ell\).
From the preceding inequalities, analysis of the summands in the above binomial expansion shows that \(K^{(\ell)}(1)>0\) and similarly
$K^{(\ell - 1)}(1) \ge 0$. 
We conclude that
$K^{(\ell)}(1) + \ell K^{(\ell-1)}(1) \neq 0$ as required. \qedhere
\end{enumerate}
\end{proof}

\begin{proposition}\label{prop:minimality111_110_101}
A solution to recurrence \eqref{eq: lowDegreeForm111}, \eqref{eq: lowDegreeForm110}, or
\eqref{eq: lowDegreeForm101} is minimal if and only if left limits of the corresponding generating function
$\mathcal{F}$ and its derivatives exist at $1/\Lambda$.
Hence the decidability of the minimality problem in these instances reduces
to checking the equalities in \autoref{lem:FallDerivativesExistCharacterisation}.
\end{proposition}
\begin{proof}
We have established that, for all the recurrences, a dominant solution defines a generating
function for which some derivative tends to $\pm \infty$ as $x \to 1/\Lambda -$. This is evident from
\eqref{eq:solutionSeries} for the recurrences \eqref{eq: lowDegreeForm111} and
\eqref{eq: lowDegreeForm110}, as a minimal solution defines a generating function which is
analytic at $1/\Lambda$. For \eqref{eq: lowDegreeForm101}, this is established in
\autoref{cor:dominantSolutionToInfinity}.

On the other hand, \autoref{lem:FallDerivativesExistCharacterisation} shows that there exists
a solution for which $\mathcal{F}$ and all its derivatives exist at $1/\Lambda$. We conclude
that this choice must correspond to a minimal solution. Hence, minimality reduces to checking
\autoref{lem:FallDerivativesExistCharacterisation}.
\end{proof}

The equalities in \autoref{lem:FallDerivativesExistCharacterisation} involve
checking whether a period, exponential period, or a period-like integral equals $0$ as follows.  We conclude that \autoref{thm:MinimalityRedToPeriodLike} holds for solutions of recurrences of
the form \eqref{eq: lowDegreeForm111}, \eqref{eq: lowDegreeForm110}, and
\eqref{eq: lowDegreeForm101}. (Notice that it is decidable whether the second equality in \autoref{lem:FallDerivativesExistCharacterisation} holds or not.)
%
\begin{itemize}

\item For \eqref{eq: lowDegreeForm111}, the equation can be rearranged to obtain
%
\begin{equation*}
\int_{0}^{1/\Lambda} \frac{1}{\Lambda^{\nu} |\lambda|^{\nu'}}
(1-\Lambda x)^{\nu} |1-\lambda x|^{\nu'} t(x) - \sum_{n  \leq -\nu } c_n(1-\Lambda y)^{n+\nu - 1} + \sum_{n < -\nu} \frac{c_n}{n+ \nu}\, dy = 0,
\end{equation*}
where $t(y)$ is an algebraic function, $\lambda$ and $\Lambda$ are the characteristic roots of
the recurrence. The parameters $\nu$ and $\nu'$ are algebraic numbers of degree at most $2$.%
\footnote{We have $\nu, \nu' \in \Q(\Lambda)$, and they can be rational even if $\Lambda$ is not.}
If they are rational, then the integral is a period
and the parameters $c_n$ are algebraic numbers.
If $\nu$ and $\nu'$ are irrational, then the integral is period-like:
the parameters $c_n$ are algebraic multiples of derivatives of
$|1-\lambda x|^{\nu'}t(x)(1-\Lambda x)$ evaluated at $1/\Lambda$,
i.e., are algebraic multiples of algebraic numbers to algebraic powers.
Hence the integral on the left is period-like.

\item For the recurrence \eqref{eq: lowDegreeForm110}, the equation can be rearranged to
\begin{equation*}
\int_{0}^{1/\beta_1} \eu^{\nu' y} (\beta_1^{-1} - y)^{\nu}t(y) - \sum_{n  \leq -\nu } c_n(1-\beta_1 y)^{n+\nu - 1} + \sum_{n < -\nu} \frac{c_n}{n+ \nu} \, dy = 0,
\end{equation*}
where $t(y)$ is an algebraic function. The integral is an exponential period, as the parameters
$\nu$ and $\nu'$ are rational. In this case the numbers $c_n$ are exponential periods, as they
are algebraic multiples of $\eu^{\nu'}$: they are derivatives of
$\eu^{\nu' y}t(y)(1-y)$ evaluated at $1$. The integral on the left is an exponential period.

\item For \eqref{eq: lowDegreeForm101}, the equation can be rearranged to obtain
\begin{equation*}
\int_{0}^{1}(1-y)^{\nu - 1} (1+y)^{\nu'-1} y^{\alpha-1}(Ay - \alpha) - \sum_{n  \leq -\nu } c_n(1-y)^{n+\nu - 1} + \sum_{n < -\nu} \frac{c_n}{n+ \nu}  \, dy = 0.
\end{equation*}
The parameters $\nu$ and $\nu'$ here are algebraic numbers, as they are
in $\Q(\beta_0)$ and $\beta_0$ is a real algebraic number. If they are
rational, then the integral is a period, and the numbers $c_n$ are algebraic. If $\nu$ and
$\nu'$ are irrational, then the integral is period-like and the numbers $c_n$ are
algebraic multiples of algebraic numbers to algebraic powers.
The integral of the left is thus period-like.
\end{itemize}

\section{Conclusions}
In the Minimality Problem we are faced with the problem of comparing the ratio of the first
two terms of a solution against the value of a polynomial continued fraction. The problem
becomes trivial if the value of the polynomial continued fraction is transcendental, as no real algebraic solution is
then minimal. For degree-1 holonomic sequences, these values can often be expressed using
hypergeometric functions: \cite[\S 6.4]{LW1992} expresses such a value as a quotient of two
(contiguous) hypergeometric functions with real algebraic parameters evaluated at real
algebraic points. A characterisation of transcendence for such numbers is not known. For
holonomic sequences of higher degree, we see values of associated polynomial continued fraction using hypergeometric
functions in \autoref{ex:criticalRatioValues}. One can obtain several transcendental
values $1-1/\xi$ for the associated continued fraction:
using the construction, we see that possible values of $\xi$ include
\begin{itemize}
\item $\eu^k = \hypgeo{0}{0}(; ; k)$ for $k\in \Q$;
\item $\cos(k) = \hypgeo{0}{1}(;1/2; -k^2/4)$ for $k \in \Q$;
\item $\log(1 + k)/k = \hypgeo{2}{1}(1,1;2;-k)$ for $k \in \Q$, $|k|<1$; and
\item $\zeta(s) = \hypgeo{\mathnormal{s}+1}{\mathnormal{s}}(1,\ldots,1; 2,\ldots,2; 1)$, for $s\in \N$, $s\geq 2$, where $\zeta(s):= \sum_{n=1}^{\infty}n^{-s}$.
\end{itemize}
(The first three equalities can be seen from the classical Taylor series expansions of the functions.
The last equality follows after cancellations.)
Many of these values are known to be transcendental. For the Riemann
zeta function, Euler proved that if \(s\) is a positive integer then \(\zeta(2s)\) is a
rational multiple of \(\pi^{2k}\) and so it follows that \(\zeta(2k)\) is transcendental. The
arithmetic study of the values of \(\zeta(2k+1)\) is a major undertaking. For example,
{Ap\'{e}ry's constant} \(\zeta(3)\) is irrational \cite{apery1979irrationalite} (not known to be transcendental) and research
has shown that infinitely many values of \(\zeta(2k+1)\) are irrational
\cite{Rivoal2000LaFonctionZetaDeRiemann}.
For the time being, there is a lack of understanding of the arithmetic properties of
values of hypergeometric functions with rational (or
algebraic%
\footnote{Notice that
the parameters of the hypergeometric functions found in \autoref{ex:criticalRatioValues} can be
algebraic.}) parameters evaluated at rational (or algebraic) points, though the study
has spanned several decades and several striking results have been established (see, e.g.,
\cite{Wolfart1988Werte, BeukersH1989Monodromy,Beukers2010Algebraic,
Gorelov2013Algebraic_independence} and references therein).

In general, establishing transcendence of a number is a very challenging task, while establishing irrationality can be easier. This aspect in mind, let us
restrict our consideration of holonomic sequences to those with rational elements. 
The following proposition shows that $\textsc{Minimality}(0,1,0)$ becomes a trivial problem under this restriction.

\begin{proposition}
If $\seq{u_n}{n}$ is a minimal solution to recurrence \eqref{eq:010Recurrence} then $u_0 / u_{-1}$ irrational.
\end{proposition}
\begin{proof}
Normalise the recurrence relation as follows.
First, multiply through by $\alpha\in\N$ to obtain the recurrence
$\alpha u_{n} = (\beta_1' n + \beta_0') u_{n-1} + \gamma_0' u_{n-2}$ with \(\alpha,\beta_1',\beta_0',\gamma_0'\in\Z\).
Second, use the transformation described in \eqref{eq: sform} to obtain a recurrence of the form
$v_n = b_n v_{n-1} + a_n v_{n-2}$ where $b_n =\beta_1' n + \beta_0' \in \Z[n]$
and $a_n = \alpha\gamma_0' \in \Z$ is constant.
Without loss of generality we may assume that each $b_n$ is positive by considering the
sequence $\seq{(-1)^n v_n}{n}$ if necessary.
The continued fraction
$\KF_{n=1}^{\infty} ({\alpha\gamma_0'}/{b_n})$ associated to minimal solutions \(\seq{v_n}{n}\) of the transformed recurrence satisfies $\alpha\gamma_0' < b_n + 1$ for sufficiently large \(n\in\N\).
Continued fractions of this form converge to irrational values (see
\cite[\S XXXIV,~pp.~512--514]{chrystal1959algebra}).
This is sufficient to prove the result because
the bijection between solution sequences \(\seq{u_n}{n}\) and \(\seq{v_n}{n}\) preserves rationality and minimality. 
\end{proof}

A similar conclusion holds for certain instances of
$\textsc{Minimality}(j,1,\ell)$ with $j + \ell = 1$ (precisely when
$|g_3(n-1)g_1(n)| < |g_2(n)| + 1$ assuming $g_i \in \Z[n]$). Again, for such recurrences the restricted Minimality Problem becomes trivial.

Let us pursue this line of thought and discuss the restricted Minimality Problem for degree-$1$ holonomic sequences in general.
Consider then recurrences of the form \eqref{eq: lowDegreeForm111}, where we understand
that the roots are of distinct moduli. We may invoke a conjecture of Zudilin
\cite{zudilin2002remarks} (see also \autoref{conj:Zudilin} for the precise statement)
that makes the following prediction: 
if a second-order Poincar\'{e} recurrence relation, in which the coefficients are in $\Q(n)$,
has irrational characteristic roots,
then all rational solutions to the recurrence are dominant.
If the conjecture is true
then the restricted Minimality Problem can be trivially answered for such recurrences 
with irrational characteristic roots.
Hence, the only interesting instances are the
recurrences which have rational characteristic roots. 
By the discussion at the end of
\autoref{subsec:Minimality111_110_101}, the Minimality Problem of rational solutions to such
recurrences reduces to checking whether a period is equal to $0$. 
This is conjectured to be decidable by Kontsevich and Zagier \cite{kontsevich2001periods} (see \autoref{conj:KZ} for the precise statement and discussion).

Consider then recurrences of the form \eqref{eq: lowDegreeForm111RepeatedRoots}.
In this case the (unrestricted) Minimality Problem reduces to checking whether an exponential period
is zero (\autoref{prop:111repeatedRoot}), which is also conjectured to be decidable
\cite{kontsevich2001periods}.
We conclude that $\textsc{Minimality}(1,1,1)$ restricted to rational
solution sequences is decidable subject to Zudilin's conjecture and the aforementioned
conjectures on periods and exponential periods.

We may say something a bit stronger.
From the interreductions between the different parametrized versions of the Minimality Problem established in \autoref{subsec:interreductionsLowDegrees},
and subject to the aforementioned conjectures by Kontsevich and Zagier, and Zudilin, restrictions of the problems
$\textsc{Minimality}(j,k,\ell)$ with $j,k,\ell \in \{0,1\}$ are decidable except for
the case $j = \ell = 1$, $k=0$.

\newpage
\appendix

\section{Conjectures}\label{app:Conjectures}

\subsection{A rationality conjecture for holonomic sequences}

The following conjecture is due to to Zudilin \cite{zudilin2002remarks}:

\begin{conj} \label{conj:Zudilin}
Suppose that
\(u_{n} = b_nu_{n-1} + a_nu_{n-2}\)
is a second-order Poincar\'{e} recurrence
with \(a_n,b_n\in\mathbb{Q}(n)\).
Assume that the characteristic roots \(\lambda\) and \(\Lambda\)
associated to the recurrence satisfy \(0<|\lambda|<|\Lambda|\). 
Suppose that there 
exist two rational linearly independent solutions \(\seq{v_n}{n}\) and
\(\seq{w_n}{n}\) satisfying
\(v_{n+1}/v_n \sim \lambda\) and \(w_{n+1}/w_n \sim \Lambda\) as
\(n\to\infty\). 
Then \(\lambda\) and \(\Lambda\) are rational numbers.
\end{conj}


\subsection{A decidability conjecture for periods}
Kontsevich and Zagier's seminal paper \cite{kontsevich2001periods} defines a \emph{period} to be a complex number whose real and imaginary parts 
can be written as absolutely convergent integrals of the form
	\begin{equation*}
		\int_\sigma \frac{P(x_1,\ldots, x_n)}{Q(x_1,\ldots, x_n)}\, dx_1 \cdots dx_n
	\end{equation*}
	where \(P,Q\in\mathbb{Q}[x_1,\ldots, x_n]\), \(Q\) is not the zero polynomial, and the domain \(\sigma\subset \R^n\) is given by polynomial inequalities with rational coefficients.
It can be shown that one can replace rational numbers by algebraic numbers, and rational functions
by algebraic functions (with algebraic coefficients) in the above definition.
The set of periods \(\mathcal{P}\) form a countable sub-algebra of \(\C\) and it is easily seen that \(\overline{\mathbb{Q}} \subset \mathcal{P} \subset \C\).  Two initial examples are:
	\begin{equation*}
		\log(\alpha) = \int_1^\alpha \frac{1}{x}\, dx \quad \text{with } \alpha\in\overline{\Q} \qquad \text{and} \qquad \pi = \int_{x^2+y^2 \le 1} \, dxdy.
	\end{equation*} 
Given two algebraic numbers \(\alpha\) and \(\beta\), the problem of determining algorithmically whether \(\alpha=\beta\) is known to be decidable. 
The decidability of the equality of two periods---that is, a decision
procedure determining whether two periods (given by two explicit
integrals) are equal---is currently open. 
The next conjecture, see \cite[Conjecture 1]{kontsevich2001periods},
by Kontsevich and Zagier, would entail that equality of periods is decidable.

\begin{conj} \label{conj:KZ}
 Suppose that a period has two integral representations. 
One can pass between the representations via a finite sequence of admissible transformations where each transformation preserves the structure that all functions and domains of integration are algebraic with coefficients in \(\overline{\Q}\).
The admissible transformations are: linearity of the integral, a change of variables, and Stokes's formula.
\end{conj}

It is currently not known whether Euler's number \(\eu\) is a period.
The following notion of exponential period was introduced in \cite{kontsevich2001periods} to extend the definition of period to a larger class containing \(\eu\).  
An \emph{exponential period} is a complex number that can be written as an absolutely convergent integral of the form
	\begin{equation*}
		\int_\sigma \eu^{-f(x_1,\ldots, x_n)} g(x_1,\ldots, x_n)\, dx_1\cdots dx_n
	\end{equation*}
where \(f\) and \(g\) are algebraic functions with algebraic coefficients and the domain \(\sigma\subset \R^n\) is a semi-algebraic set defined by polynomials with algebraic coefficients.
%
%
%
%
\autoref{conj:KZ} is predicted to generalise to exponential periods in \cite{kontsevich2001periods}.
An overview discussing both periods and exponential periods can be found in \cite{waldschmidt2006trans}.

In this paper we encounter integrals that generalise the above concepts of period and exponential period.
A \emph{period-like integral} is a number that can be written as an absolutely convergent integral of the form
\begin{equation*}
 \int_\sigma \eu^{-f(x_1,\ldots, x_n)} g(x_1,\ldots, x_n)\, dx_1\cdots dx_n.
\end{equation*}
Here \(f\) is an algebraic function, \(g\) is the sum of algebraic functions raised to algebraic powers, and the domain \(\sigma\subset \R^n\)
is a semi-algebraic set defined by polynomials with algebraic coefficients.
\section{The PCF Equality Problem and the Minimality Problem}
\label{app:PCFmin}

\autoref{cor:PCFmin} is a straightforward application of Pincherle's Theorem (\autoref{thm: pincherle}). 
Given a solution sequence \(\seq[\infty]{v_n}{n=-1}\) 
to recurrence relation~\eqref{eq:lowDegreeOGForm}, 
let us consider the corresponding sequence \(\seq[\infty]{u_n}{n=-1}\) to the normalised recurrence (using the transformation described for \eqref{eq: sform}). This transformation preserves minimality so that \(\seq{v_n}{n}\) is a minimal solution of \eqref{eq:lowDegreeOGForm} if and only if \(\seq{u_n}{n}\) with initial terms  $u_{-1}=v_{-1}$ and $u_0=g_3(0)v_0$  is a minimal solution of \eqref{eq: sform}.
The sequence \(\seq{u_n}{n}\) is associated to the polynomial continued fraction \(\KF(a_n/b_n)\) with partial quotients $b_n= g_2(n)$ and $a_n=g_1(n)g_3(n-1)$ for each \(n\in\N\).
By \autoref{thm: pincherle}, \(\seq{u_n}{n}\) is a minimal solution to \eqref{eq: sform} if
\(\KF(a_n/b_n)\) converges to the limit $-u_0/u_{-1}$. 
Thus if one can determine the value of a polynomial continued fraction then one can determine whether \(\seq{u_n}{n}\) is a minimal solution of \eqref{eq: sform}. 
It follows that one can decide whether \(\seq{v_n}{n}\) is a minimal solution of \eqref{eq:lowDegreeOGForm}, as desired.
\gk{Can you comment on the case that \(\KF(a_n/b_n)\to\infty\) here?  The oracle should return \emph{no} for every real value, but the canonical denominators form a minimal solution sequence.}

Conversely, given a polynomial continued fraction \(b_0 + \KF(a_n/b_n)\) and a real-algebraic number $\xi$, let us construct the holonomic sequence \(\seq[\infty]{u_n}{n=-1}\) such that for each \(n\in\N\), $u_n = b_n u_{n-1} + a_n u_{n-2}$ with initial conditions $u_{-1}=1$ and 
$u_0= b_0-\xi$.
By \autoref{thm: pincherle}, 
sequence \(\seq{u_n}{n}\) is a minimal solution of the recurrence relation if and only if the continued fraction \(\KF ({a_n}/{b_n})\) converges to the value $-u_0/u_{-1}= \xi - b_0$. 
So given a holonomic sequence, if one can determine whether the sequence is a minimal solution of the associated recurrence relation then one can test the value of a polynomial continued fraction.

\section{Complex characteristic roots}\label{app:ComplexCharacteristicRoots}

We study the recurrence relation \eqref{eq:lowDegreeOGForm} 
under the assumptions that 
the recurrence relation has signature \((+,-)\)
and the discriminant
\(\Delta(n) = g_2(n)^2 + 4g_3(n)g_1(n)<0\) for each \(n\in\N\).
Our aim is to establish \autoref{prop:noPositiveMinimalSolution}. 
%
%
%



Let $\seq[\infty]{u_n}{n=-1}$ be a non-trivial solution to recurrence \eqref{eq:lowDegreeOGForm} and 
$\seq[\infty]{x_n}{n=0}$ be the associated sequence with terms \(x_n = u_n/u_{n-1}\)
consider the 
function \(f_n\colon\R\setminus\{0\}\to\R\) given by
\(f_n(x)= g_3(n)x - g_2(n) - g_1(n)/x\).
Observe that \(f_n\) is continuous and has no real roots since $\Delta(n)<0$.
Furthermore, we have
\(\lim_{x\to 0^+} f_n(x) = \lim_{x\to\infty} f_n(x) = \infty\).  We thus conclude that for each \(n\in\N\), \(f_n\) is a strictly positive function on $(0,\infty)$.

\begin{lemma} \label{lem: xformulation}
We have that \(x_n = x_{n-1} - f_{n}(x_{n-1})/g_3(n)\) for each \(n\geq 1\).
Moreover, we have \(x_n = x_0 -\sum_{j=1}^n f_j(x_{j-1})/g_3(j)\). Thus, if $\seq[\infty]{x_n}{n=0}$
is a positive sequence, then it is strictly decreasing.
\end{lemma}

\begin{proof}
Substitution shows that \(f_n(x_{n-1}) = g_3(n)(x_{n-1} - x_n)\), 
and further we have
\begin{equation*}
x_n - x_0 = \sum_{j=1}^n x_j - x_{j-1} = - \sum_{j=1}^n f_j(x_{j-1})/g_3(j).
\end{equation*}
 It is now evident that the sequence $\seq[\infty]{x_n}{n=0}$ is strictly decreasing since both
\(g_3(j)\) and \(f_j(x_{j-1})\) are strictly positive for each \(j\in\N\) under the assumption
that $\seq[\infty]{x_n}{n=0}$ is positive.
\end{proof}

We define the functions $h_0(x)= f_0(x)$ and 
$h_\infty(x) = \alpha_1x - \beta_1 - \gamma_1/x = \lim_{n\rightarrow\infty} {f_n(x)}/n$. 
Note that \(f_n(x)= h_0(x) + nh_\infty(x)\) and the two functions \(h_0(x)\) and \(h_\infty(x)/x\) are differentiable and non-negative in the domain \(\{x\in\R : x>0\}\).
%
The function
\(h_0(x)\) has no real roots, 
and \(h_\infty(x)/x\) has a single real root if \(\beta_1^2 + 4\alpha_1\gamma_1 =0\) and no root otherwise.
By continuity, it follows that there exists an \(\varepsilon_0>0\) such that for all \(\{x\in\R : x>0\}\), \(h_0(x)/x > \varepsilon_0\).

\begin{proof}[Proof of \autoref{prop:noPositiveMinimalSolution}]
Let $\seq[\infty]{u_n}{n=-1}$ be a non-trivial positive solution.
If there is an \(N\in\N\) such that \(u_N = 0\) then it is clear that a subsequent term is negative and so we can assume that \(u_n>0\) for each \(n\in\{-1,0,1,\ldots\}\).
Thus for each \(n\in\N_0\), \(x_n= u_n/u_{n-1}>0\).
Since \(h_0(x)\) and \(h_\infty(x)\) are both non-negative on the domain \(\{x\in\R : x>0\}\) 
and, in addition, there exists an \(\varepsilon_0>0\) such that \(h_0(x)/x > \varepsilon_0\), 
we have that \(f_j(x_{j-1}) = h_0(x_{j-1}) + jh_\infty(x_{j-1})>\varepsilon_0\) too.
We combine this uniform bound and \autoref{lem: xformulation} to obtain
\begin{equation*}
   x_n = x_0 - \sum_{j=1}^n \frac{f_j(x_{j-1})}{g_3(j)} \le x_0 - 
   \sum_{j=1}^n \frac{\varepsilon_0}{\alpha_1 j + \alpha_0}.
\end{equation*}
Since the harmonic series diverges, we deduce that there exists an \(N\in\N\) such that \(x_N<0\), a contradiction. 
It follows that \(\seq{u_n}{n}\) is not positive.
\end{proof}

\section{Testing the initial ratio}
\label{ap:compxtomu}

The goal of this section is to prove \autoref{prop:compxtomu}:
given a non-minimal solution \(\seq[\infty]{u_n}{n=-1}\) to recurrence relation~\eqref{eq:lowDegreeOGForm} with a $(+,-)$ signature and positive discriminants,
decide if \(\seq{u_n}{n}\) is positive.
By \autoref{cor:MinimalMuPositiveMu}, an equivalent problem is to decide if \(u_0/u_{-1} \ge \mu\).
Together \autoref{cor:MinimalMuPositiveMu} and \autoref{lem:expand} determine a computable threshold for the sequence of ratios \(\seq[\infty]{x_n}{n=0}\) associated with \(\seq[\infty]{u_n}{n=-1}\)
such that \(\seq{x_n}{n}\) crosses this threshold if and only if $u_0/u_{-1} = x_0> \mu$.
Note that an upper bound on the number of steps taken in computing this threshold depends on the distance $|u_0/u_{-1} - \mu|$.

We define the \(n\)th characteristic polynomial \(\chi_n\) for recurrence relation~\eqref{eq:lowDegreeOGForm} as \(\chi_n(x) = g_3(n)x^2 - g_2(n)x - g_1(n)\) for each \(x\in \R\). 
In this section we shall assume that the associated sequences of characteristic roots \(\seq[\infty]{\lambda_n}{n=1}\) and \(\seq[\infty]{\Lambda_n}{n=1}\) are both real. 
Let $\lambda_\infty$ and $\Lambda_\infty$ be the corresponding limits, if defined, of these sequences%
\footnote{In the case of Poincar\'{e} recurrences, $\lambda_{\infty}$ and
$\Lambda_{\infty}$ coincide with the roots $\lambda$ and $\Lambda$ of the associated characteristic polynomial.}. 
Note that, from the closed form of \(\seq[\infty]{\lambda_n}{n=1}\) associated to recurrence
relations considered in this section,
one can observe that the limit $\lambda_\infty$ is always finite and thus well-defined 
(which is not the case of the limit for $\Lambda_\infty$). If \(\seq[\infty]{\Lambda_n}{n=1}\)
diverges, we choose $\Lambda_\infty = +\infty$.

\subsection{Monotonicity and the characteristic roots}

The threshold described in the opening of this section depends on the monotonicity of the associated sequences of \(n\)th characteristic roots.

\begin{lemma} \label{monoroots}
 The sequences \(\seq[\infty]{\lambda_n}{n=1}\) and \(\seq[\infty]{\Lambda_n}{n=1}\) are eventually monotonic.
\end{lemma}
\begin{proof}
Let us define a function \(\lambda \colon \R\to\R\) given by
\begin{equation*}
\lambda(x) = \frac{g_2(x) - \sqrt{g_2(x)^2 + 4g_1(x)g_3(x) }}{2g_3(x)}.
\end{equation*}
Note that for each \(n\in\N\), \(\lambda(n) = \lambda_n\).
Then we can write the derivative of the function in terms of constants
$A, B$ and $C$ (see \cite{Liu2010positivity}) as follows:
\begin{align*}
\lambda'(x) = & -\frac{g_2(x) C +2g_3(x) B - C \sqrt{g_2(x)^2+4g_3(x)g_1(x)}}{2g_3(x)^2\sqrt{g_2(x)^2-4g_3(x)g_1(x)}}\\
= & - \frac{2(B^2-AC)}{\sqrt{g_2(x)^2+g_3(x)g_1(x)}(g_2(x)C+2g_3(x)B+C\sqrt{g_2(x)^2+g_3(x)g_1(x)})}.
\end{align*}
From the above equations, we have the following cases:
\begin{itemize}
	\item If $C\geq 0$ and $g_2(x)C+2g_3(x)B\geq 0$, then $\sign(\lambda'(x)) = -\sign(B^2-AC) $. 
	\item If $C\leq 0$ and $g_2(x)C+2g_3(x)B\geq 0 $, then $\sign(\lambda'(x)) =-1$. 
	\item If $C\geq 0$ and $g_2(x)C+2g_3(x)B\leq 0 $, then $\sign(\lambda'(x)) =1$. 
	\item If $C\leq 0$ and $g_2(x)C+2g_3(x)B\leq 0 $, then $\sign(\lambda'(x)) = \sign(B^2-AC) $. 
\end{itemize}

The sign of $g_2(x)C+2g_3(x)B$ changes at most once. 
Thus, the sign of $\lambda'$ is eventually constant and therefore 
$\lambda\colon \R\to\R$ is eventually monotonic. 
It follows that \(\seq[\infty]{\lambda_n}{n=1}\) is eventually monotonic.
A similar argument proves that \(\seq[\infty]{\Lambda_n}{n=1}\) is eventually monotonic. 
\end{proof}

\subsection{A threshold for positivity}

Let  \(\seq[\infty]{u_n}{n=-1}\) be a solution sequence of the recurrence with a $(+,-)$ signature
\eqref{eq:lowDegreeOGForm} and let \(\seq[\infty]{x_n}{n=0}\) be its associated sequence of ratios. 
Without loss of generality, we shall assume that the sequence \(\seq[\infty]{\lambda_n}{n=1}\) is monotonic.
We first consider the case where \(\seq[\infty]{\lambda_n}{n=1}\)
is decreasing. 
%
\begin{proposition}
	\label{prop:xgeqlambn_inc}
Suppose that $\seq[\infty]{\lambda_{n}}{n=1}$ is decreasing, that there exists a $k\in \nats$ 
such that $x_{k} \geq \lambda_{k}$ and that \(u_{-1}, u_0, \ldots, u_{k}>0\), 
then \(\seq[\infty]{u_n}{n=-1}\) is positive.
\end{proposition}
\begin{proof}
From the assumptions and the recurrence relation for \(\seq[\infty]{x_n}{n=0}\), we obtain the following inequalities:
\(
   g_3(k+1)x_{k+1} \ge  g_2(k+1) + g_1(k+1)/\lambda_{k} \ge g_3(k+1)\lambda_{k+1}
\)
and so \(x_{k+1}\ge \lambda_{k+1}\).  It follows by induction that \(x_n\ge \lambda_n >\lambda_\infty\) for all \(n>k\).  Thus \(\seq[\infty]{u_n}{n=-1}\) is positive.
\end{proof}

\noindent We obtain a similar threshold for positivity when \(\seq[\infty]{\lambda_n}{n=1}\) is 
increasing.

\begin{proposition}
	\label{prop:xgeqlambinf_inc}
Suppose that $\seq[\infty]{\lambda_n}{n=1}$ is increasing, \(\lambda_\infty<\Lambda_\infty\), there exists $k\in \nats$ such that 
$x_{k} \geq \lambda_\infty$ and that \(u_{-1}, u_0, \ldots, u_{k}>0\), 
then \(\seq{u_n}{n}\) is positive.
\end{proposition}
\begin{proof}
We can assume without loss of generality that $\lambda_\infty < \Lambda_n$. 
As a consequence, $\lambda_\infty\in [\lambda_n, \Lambda_n)$, and so we have 
$g_3(n)\lambda_\infty \leq g_2(n) + g_1(n)/\lambda_\infty$.
From this result and the existence of \(k\in\N\) such that \(x_{k}\ge \lambda_\infty\),
we have
\(
   g_3(k+1)x_{k+1} \ge  g_2(k+1) + g_1(k+1)/\lambda_\infty \ge g_3(k+1)\lambda_\infty
\)
and so \(x_{k+1}\ge \lambda_\infty\).  
It follows by induction that \(x_n\ge \lambda_\infty\) for all \(n>k\).  Thus \(\seq[\infty]{u_n}{n=-1}\) is positive.
\end{proof}

The case when we have a single repeated characteristic root \(\lambda_\infty=\Lambda_\infty\) is more involved.

\begin{proposition}
	\label{prop:xgeqlambinf_inceqroot}
Suppose that recurrence~\eqref{eq:lowDegreeOGForm} has a single repeated characteristic root.  
Let us assume that $\seq[\infty]{\lambda_{n}}{n=1}$ is increasing, there exists an $k\in \nats$ such that $x_{k} \geq \sqrt{{-g_1({k+1})}/{g_3({k+1})}}$ or 
$x_{k} \geq g_2({k})/({2g_3({k}}))$, 
and \(u_{-1}, u_0, \ldots, u_{k}>0\).  Then \(\seq[\infty]{u_n}{n=-1}\) is positive.
\end{proposition}
\begin{proof}
Consider the constant $C = \alpha_0  \beta_1 - \alpha_1\beta_0$.
We start with the case $C<0$. In this case the sequence with terms given by \(g_2(n)/g_3(n)\) is decreasing as
  \begin{equation*}
  \frac{g_2(n+1)}{g_3(n+1)} - \frac{g_2(n)}{g_3(n)} = \frac{C}{g_3(n)g_3(n+1)},
  \end{equation*}
and additionally \(g_2(n)/(2g_3(n)) \ge \lambda_\infty$. 
We obtain the following inequalities using our assumption on $k\in \nats$:
\begin{equation*}
x_{k+1} \geq \frac{g_2(k+1)}{g_3(k+1)} - \frac{\sqrt{-g_3(k+1)g_1(k+1)}}{g_3(k+1)}
\geq  \frac{g_2({k+1})}{2g_3({k+1})}
\geq \lambda_\infty.
\end{equation*}
The result in this case follows similarly to the method outlined in \autoref{prop:xgeqlambinf_inc}.

Consider the case $C\geq 0$. 
Let us show that for all $n\geq k+1$ we have $x_n \geq {g_2(n)}/({2g_3(n)})$.
We outline the inductive step of the proof.  
Suppose that  $n\geq k+1$ and assume the inductive hypothesis holds for $n$.  Then
\begin{equation*}
x_{n+1} 
\geq \frac{g_2({n+1})}{2g_3({n+1})} + \frac{g_2({n+1})g_2(n) + 4g_1({n+1})g_3(n)}{2g_2(n)g_3({n+1})}
\geq \frac{g_2({n+1})}{2g_3({n+1})} + \frac{\beta_1\beta_0 + 4\alpha_0\gamma_1}{2g_2(n)g_3(n+1)}.
\end{equation*}
As $C\geq 0$, we have that $\beta_0 \geq \beta_1\alpha_0/{\alpha_1}$.
Thus we obtain
  \begin{equation*}
   x_{n+1} \ge \frac{g_2({n+1})}{2g_3({n+1})} + \frac{\beta_1\beta_0 + 4\alpha_0\gamma_1}{2g_2(n)g_3(n+1)} \ge \frac{g_2({n+1})}{2g_3({n+1})},
  \end{equation*}
which concludes the induction step.
It follows that the sequence $\seq{x_n}{n}$, and so $\seq{u_n}{n}$ remains positive.
\end{proof}

Let $\seq[\infty]{\mu_n}{n=0}$ denote the sequence of ratios associated to a solution to recurrence~\eqref{eq:lowDegreeOGForm} with initial ratio $\mu_0=\mu$.  We have the following:
\begin{lemma}
\label{lem:expand}
Let \(\seq[\infty]{u_n}{n=-1}\) be a solution to recurrence~\eqref{eq:lowDegreeOGForm} and 
\(\seq[\infty]{x_n}{n=0}\) the sequence of consecutive ratios.  Suppose that there exists 
$\varepsilon>0$ such that $x_0>\mu + \varepsilon$.  Then  
for each \(n\in\N\), we have the following results.
\begin{enumerate}
\item If $\seq[\infty]{\lambda_n}{n=1}$ is decreasing, then for all $n\in \N$, either 
$x_n> \mu_n + \varepsilon$ or $x_n\geq \lambda_n$.
\item If $\seq[\infty]{\lambda_{n}}{n=-1}$ is increasing, then for all $n\in \N$ one of the 
following occurs: $x_n> \mu_n + \varepsilon$, $x_n\geq \lambda_\infty$, 
$x_n \geq{g_2({n})}/({2g_3({n})})$ or $x_n\geq \sqrt{{-g_1({n+1})}/{g_3({n+1})}}$.
\end{enumerate}
\end{lemma}
\begin{proof} \hfil
\begin{enumerate}
\item Suppose that $\seq[\infty]{\lambda_{n}}{n=1}$ is decreasing. We proceed by induction. 
The base case is given by hypothesis. Assume the induction hypothesis holds for $n\in \N$.
If $x_n \geq \lambda_{n}$, then, as in the proof of Proposition~\ref{prop:xgeqlambn_inc},
$x_{n+1}\geq \lambda_{n+1}$. 
Similarly, if $x_n \geq \lambda_{n+1}$, we have
\begin{equation*}
g_3(n+1)x_{n+1} = g_2({n+1}) + \frac{g_1({n+1})}{x_n} \geq g_2({n+1}) + 
\frac{g_1({n+1})}{\lambda_{n+1}}\geq g_3({n+1})\lambda_{n+1}
\end{equation*}
and so
$x_{n+1}\geq \lambda_{n+1}$. 
Otherwise, we have the following inequalities:
\begin{equation*}
x_{n+1}-\mu_{n+1} =\frac{g_1({n+1})}{x_ng_3({n+1})}-\frac{g_1({n+1})}{\mu_ng_3({n+1})}
> \frac{-g_1({n+1})\varepsilon}{g_3({n+1})x_n\mu_n}
> \frac{-g_1({n+1})\varepsilon}{g_3({n+1})\lambda_{n+1}^2}
> \varepsilon.
\end{equation*}
The last inequality holds since $\chi_{n+1}(\sqrt{-g_1(n+1)/g_3(n+1)})<0$ when 
$\Delta({n+1})\geq 0$.

\item Suppose that $\seq[\infty]{\lambda_{n}}{n=1}$ is increasing. 
The respective inductive proofs for the $x_n\geq \lambda_\infty$, 
$x_n\geq g_2({n})/{2g_3({n})}$
and $x_n\geq \sqrt{-g_1({n+1})/{g_3({n+1})}}$ follow by Propositions~\ref{prop:xgeqlambinf_inc} and \ref{prop:xgeqlambinf_inceqroot}. 
Otherwise, as before, we have:
\begin{equation*}
x_{n+1}-\mu_{n+1} > \frac{-g_1({n+1})\varepsilon}{g_3({n+1})x_n\mu_n} \geq \varepsilon.
\end{equation*}
    \end{enumerate}
The proof is complete.
\end{proof}

\begin{proof}[Proof of \autoref{prop:compxtomu}]
Let $\seq[\infty]{y_n}{n=0}$ denote a sequence of ratios of consecutive terms of a solution 
to recurrence~\eqref{eq:lowDegreeOGForm}. If there exists $n_0\geq 0$ such that 
$y_{n_0} < \min(\lambda_\infty, \lambda_{n_0+1})$, then it can be shown that 
for all $n\geq n_0$, $y_{n} < \min(\lambda_\infty, \lambda_{n+1})$ and 
$\seq[\infty]{y_n}{n=n_0}$ is a decreasing sequence. If $\seq[\infty]{y_n}{n=0}$ is positive,
then it is converging, which is impossible as the only possible limits of such 
a sequence of ratios are $\lambda_\infty$ and $\Lambda_\infty$.

The sequence $\seq{\mu_n}{n}$ being positive, it thus satisfies for all $n\geq 0$ that 
$\mu_n \geq \min(\lambda_\infty, \lambda_{n+1})$. 
It follows from \autoref{lem:expand} 
that the positivity of a solution sequence $\seq{u_n}{n}$ and its sequence of consecutive ratios
\(\seq[\infty]{x_n}{n=0}\) is determined by one the threshold crossings 
given in Propositions~\ref{prop:xgeqlambn_inc},~\ref{prop:xgeqlambinf_inc} 
and~\ref{prop:xgeqlambinf_inceqroot}.
One can thus detect whether $x_0>\mu$ by computing an initial number of terms in the sequence 
\(\seq{x_n}{n}\). On the one hand, this algorithm is guaranteed to terminate.
On the other hand, the number of steps does not have an upper bound independent on the distance
between $x_0$ and $\mu$.
\end{proof}

\section{The Positivity and the Ultimate Positivity Problems}
\label{app:Ultpos}

In this section we establish the second statement in \autoref{th:oracleRel}: that the Positivity and Ultimate Positivity Problems in this setting are interreducible problems. 
As before we separate our discussion according to the signature of recurrence \eqref{eq:lowDegreeOGForm}.
The two degenerate cases that are solved using first-order recurrence relations are discussed in \autoref{sub:linrec}.
In the following discussion we assume that a given sequence $\seq[\infty]{u_n}{n=-1}$ satisfying recurrence \eqref{eq:lowDegreeOGForm} has $u_{-1},u_0\neq 0$.
Otherwise,
if \(u_{-1},u_0=0\) then the sequence is trivially positive, and if only one of the initial terms is zero then a suitable shift gives initial terms that are non-zero.

Let us first consider that recurrence \eqref{eq:lowDegreeOGForm} has signature \((+,-)\) and, without loss of generality, assume that \(\sign(\Delta(n))\) is constant.
Suppose that $\seq[\infty]{u_n}{n=-1}$ is a solution of \eqref{eq:lowDegreeOGForm}.
We can assume that the initial terms \(u_{-1}\) and \(u_0\) have the same sign.
For otherwise $u_0$ and $u_1$ have the same sign, so one can shift the sequence by one step 
to obtain this property.
As the first two terms have the same sign, the ratio is positive.
We can thus rely on the results 
of~\autoref{sec:posmin}.
If the sequence of discriminants is negative, we can 
deduce from~\autoref{prop:noPositiveMinimalSolution} that the sign of $\seq[\infty]{u_n}{n=-1}$
changes infinitely often and thus there are no positive nor ultimately positive sequences.
If the sequence of discriminants is positive, with an initial shift,
we can assume that for all $n\in \N, g_2(n)/g_3(n)\geq \max(\lambda_n,\lambda)$ (as $\lambda$ is
finite). 
We then have that the sequence
$\seq[\infty]{u_n}{n=-1}$ changes sign at most once. 
Indeed, if there exists $n_0\geq 0$ 
such that $u_{n_0}$ does not have the same sign as $u_{n_0-1}$, then 
$u_{n_0+1}/u_{n_0}\geq g_2(n_0+1)/g_3(n_0+1)$ and from \autoref{prop:xgeqlambn_inc}, 
\autoref{prop:xgeqlambinf_inc} and~\autoref{prop:xgeqlambinf_inceqroot}, 
this implies that the sequence of ratios will remain positive.
As the sign of the sequence changes at most once, a sequence is positive if and only if it is 
ultimately positive.

Let us now consider the case that \eqref{eq:lowDegreeOGForm} has signature \((-,+)\).
Assume first that $u_{-1},u_0>0$. Then, as shown in \autoref{prop:pos}, 
the only positive solution sequences are those that are minimal.
Moreover, as this holds
also for any shift of the sequence, the only ultimately positive sequences are those that are minimal.
Now if \(u_{-1},u_0<0\), then through similar reasoning, $\seq[\infty]{u_n}{n=-1}$
is neither positive nor ultimately positive.
Now assume that the two initial terms have opposite signs. Without loss of generality, one can 
assume that $u_0> 0$. 
Consider the sequence $\seq[\infty]{v_n}{n=-1}$ such that for all 
$n\in\N, v_n =(-1)^n u_n$. 
This sequence starts with two positive terms and
satisfies the recurrence relation
\begin{equation}
\label{eq:moinsu}
g_3(n) v_n= -g_2(n) v_{n-1} + g_1(n) v_{n-2},
\end{equation}
which has signature \((+,+)\).
We conclude that the transformed sequence \(\seq[\infty]{v_n}{n=-1}\) has constant sign,
which implies that the signs in sequence $\seq[\infty]{u_n}{n=-1}$ alternate.
Thus $\seq[\infty]{u_n}{n=-1}$ is neither positive nor ultimately positive.

Let us now consider the case the \eqref{eq:lowDegreeOGForm} has signature \((+,+)\).
Suppose that $\seq[\infty]{u_n}{n=-1}$ is a solution to recurrence~\eqref{eq:lowDegreeOGForm}.
If \(u_{-1}\) and \(u_0\) have the same sign, then trivially the sequence has constant sign.
Assume that \(u_{-1}\) and \(u_0\) have opposite signs, let us assume without loss of generality that $u_0> 0$.
Consider the sequence $\seq[\infty]{v_n}{n=-1}$ such that for each 
$n\in\N, v_n =(-1)^n u_n$. 
This sequence starts with two positive terms and
satisfies recurrence relation~\eqref{eq:moinsu}.
As seen earlier, in this case we can detect with either a positivity or an ultimate 
positivity oracle whether the sequence $\seq[\infty]{v_n}{n=-1}$ remains positive or
if its sign alternates. 
In the later case, need only determine whether the sign alternates on the
even or odd terms to decide whether $\seq[\infty]{u_n}{n=-1}$ is positive, which can be achieved by computing a finite number of terms.

It is trivial to see that there are no positive nor ultimately positive non-trivial solutions when \eqref{eq:lowDegreeOGForm} has signature \((-,-)\).

\section{Interreductions Between Degree-1 Holonomic Sequences}\label{app:interreductions}


In this short appendix, we prove the remaining cases of \autoref{prop:linearReductions}.

\begin{proof}[Proof of \autoref{prop:linearReductions}]\hfil
\begin{enumerate}
\item
Follows immediately from the interreductions of the recurrence relations \eqref{eq: ogform} and \eqref{eq: sform} described in the preliminaries.

\addtocounter{enumi}{1}

\item 
We proceed as follows. As $\alpha_1 \neq 0$, we may divide through by $\alpha_1$ if necessary. We set $\alpha:= \alpha_0/\alpha_1$. Next, if $\beta_1 \neq |\gamma_1|$,
we consider the sequence
$\seq{(\sign(\gamma_1)\beta_1/\gamma_1)^n u_n}{n}$ instead, as this sequence satisfies
the recurrence
\begin{equation*}
(n+\alpha) v_n = \sign(\gamma_1)\frac{\beta_1}{\gamma_1}(\beta_1 n + \beta_0)v_{n-1} + \frac{\beta_1^2}{\gamma_1^2}(\gamma_1 n + \gamma_0)v_{n-2}.
\end{equation*}
%
Clearly minimality is preserved in this translation and $\sign(\sign(\gamma_1)\beta_1^2/\gamma_1) = 1$.
Hence the desired result follows.

Assume now that $\beta_1^2 + 4\gamma_1 = 0$ in \eqref{eq: lowDegreeForm111}.
 It follows that $\gamma_1 = -\beta_1$ (as $|\gamma_1| = |\beta_1|$), and since
$\beta_1(\beta_1 - 4) = 0$ with $\beta_1 \neq 0$, it follows that $\beta_1 = 4$. 
Now the sequence $\seq{(1/2)^n u_n}{n}$ satisfies recurrence
\eqref{eq: lowDegreeForm111RepeatedRoots} and minimality is clearly preserved in this transformation.

\item Analogous to the first part of the above case.

\item In this case the recurrence admits a minimal solution if and only if
$\gamma_1 \alpha_1 > 0$. 
This follows by an application of \autoref{thm:contFracConvChar} with
$r(n) = 1 + 4\frac{g_1(n)g_3(n-1)}{g_2(n)g_2(n-1)} = 4\frac{\alpha_1 \gamma_1}{\beta_0^2}n^2 + o(n^2)$
The reduction to \eqref{eq: lowDegreeForm101} then follows by considering the sequence $\seq{(\sign(\beta_0)\sqrt{\alpha_1/\gamma_1})^n u_n}{n}$.
\qedhere
\end{enumerate}
\end{proof}

\section{Analytic properties of the generating function} 

\subsection{Associated differential equation}\label{app:ode}
We consider a differential equation associated to the recurrence relation
\eqref{eq:lowDegreeOGForm}.
We assume here that $\deg(a) = 1$. In particular, we have $\alpha_1,\alpha_0 > 0$. By dividing through by $\alpha_1$, we may take $\alpha_1 = 1$. By shifting, we may further assume that $\alpha:= \alpha_0 > 1$. The recurrence relation we consider is thus of the form
\begin{equation}\label{eq:reca1=1}
(n+\alpha) u_n = (\beta_1 n + \beta_0)u_{n-1} + (\gamma_1 n + \gamma_0)u_{n-2}. 
\end{equation}
We allow here $\beta_1 = 0$ or $\gamma_1 = 0$.

Consider the generating series
$\mathcal{F}(x) = \sum_{n = -1}^{\infty}u_n x^{n + \alpha}$.
By relation \eqref{eq:reca1=1}, we have that
\begin{equation*}
  \sum_{n=2}^\infty g_3(n)u_{n-1} x^{n+\alpha} = \sum_{n=2}^\infty g_2(n)u_{n-2} x^{n+\alpha} + \sum_{n=2}^\infty g_1(n)u_{n-3} x^{n+\alpha}.
\end{equation*}
Observe now that
 \begin{align*}
   \sum_{n=2}^\infty g_3(n) u_{n-1} x^{n+\alpha} &= x \sum_{n=2}^\infty (n+\alpha)u_{n-1} x^{n+\alpha-1}
    = x \mathcal{F}'(x) -(\alpha + 1) u_0 x^{\alpha+1}
    				-\alpha u_{-1} x^{\alpha}.
  \end{align*}
In a similar fashion one can write
  \begin{equation*}
   \sum_{n=2}^\infty g_2(n) u_{n-1} x^{n+\alpha} 
		= \beta_1 x^2 \mathcal{F}'(x) + (\beta_0 + (1-\alpha)\beta_1)x\mathcal{F}(x) - (\beta_0+\beta_1)u_{-1} x^{\alpha+1},
  \end{equation*}
and
  \begin{equation*}
   \sum_{n=2}^\infty g_1(n) u_{n-2} x^{n+\alpha}
   		= \gamma_1x^3\mathcal{F}'(x) + (\gamma_0+(2-\alpha)\gamma_1)x^2\mathcal{F}(x).
  \end{equation*}
When we combine the above three results we obtain the first-order differential equation \(\mathcal{F}'(x) + s(x) \mathcal{F}(x) = t(x)\)
with the functions \(s\) and \(t\) defined in \eqref{eq:sx}.
%
%

\subsection{Proof of \autoref{clm:integralAsymptotics}}\label{subapp:IntegralAsymptotics}

We will show that
\begin{equation*}
\int_{0}^{x} f(y)\, dy = \sum_{n  < -\nu}
			\frac{-c_n/\Lambda }{n+ \nu}(1-\Lambda x)^{n + \nu}
			+ C_0 + C_1 \log(1-\Lambda x) + \mathcal{O}((1-\Lambda x)^{A+1}),
\end{equation*}
where $C_1 = -c_{-\nu}/\Lambda$ if \(\nu\) is a non-positive integer and $C_1 = 0$ otherwise, and 
\begin{equation*}
C_0 = \sum_{n < - \nu } \frac{c_n/\Lambda}{n+ \nu}
 + \int_{0}^{1/\Lambda} f(y) - \sum_{n \leq -\nu } c_n(1-\Lambda y)^{n+\nu - 1}\, dy.
\end{equation*}

%
%
%
%
\begin{proof}
Let us write $r(y) = \sum_{n\leq -\nu}c_n(1-\Lambda y)^{n+\nu - 1}$.
Then
	\begin{equation*}
		\int_0^x f(y)\, dy = \int_0^{1/\Lambda} f(y)-r(y)\, dy - \int_{x}^{1/\Lambda} f(y) -r(y)\, dy + \int_0^x r(y) \, dy.
	\end{equation*}
We study the three integrals on the right-hand side.
The first integral can be written in terms of \(C_0\) as follows:
	\begin{equation*}
		\int_0^{1/\Lambda} f(y)-r(y)\, dy = \int_{0}^{1/\Lambda} f(y) - \sum_{n \leq -\nu } c_n(1-\Lambda y)^{n + \nu - 1} \, dy =
C_0 - \sum_{n < -\nu} \frac{c_n/\Lambda}{n + \nu}.
	\end{equation*}
The integrand in the second integral is analytic in its domain and so, by integrating the power series expansion, we obtain the estimate
\begin{equation*}
\int_{x}^{1/\Lambda} f(y)-r(y) \,dy =  \mathcal{O}((1-\Lambda x)^{n_0 + \nu}).
\end{equation*}
For the third integral we have
	\begin{multline*}
		\int_0^x r(y) \, dy = \sum_{n\le -\nu} \int_0^x c_n(1-\Lambda y)^{n+\nu - 1}\, dy \\
		= \sum_{n < -\nu} \frac{-c_n/\Lambda}{n+\nu}(1-\Lambda x)^{n+\nu} + C_1\log(1-\Lambda x) + \sum_{n<-\nu} \frac{c_n/\Lambda}{n + \nu}.
	\end{multline*}
In the above $C_1 = 0$ if $-\nu \notin \N_0$, otherwise $C_1 = -c_{-\nu}/\Lambda$.
Combining these three results gives the desired form.
\end{proof}

\section{Justification for \autoref{ex:KoomanAsymptotics}}\label{app:asymptotics} 

The aim of this appendix is to establish the claimed asymptotic behaviours of solutions to the
recurrence relations in \autoref{ex:KoomanAsymptotics}. The proof of this is a straightforward
application of the framework given by Kooman in \cite{kooman2007asymptotic}, but we give a
proof for the sake of completeness.

Recall that the recurrence relations in hand are
\begin{subequations}
\begin{align}
\label{eq:appRecurrences101}
(n+\alpha) v_n &= \beta v_{n-1} + (n + \gamma) v_{n-2}, \quad \text{and} \\
\label{eq:appRecurrences111}
(n + \alpha) v'_n &= (2n + \beta) v'_{n-1} - (n + \gamma) v'_{n-2}.
\end{align}
\end{subequations}
In the former recurrence, we assume $\beta > 0$, and in the latter we assume
$\beta > \alpha + \gamma$.

For the duration of this appendix, the sequence $\seq{v_n}{n}$ (resp., $\seq{v'_n}{n}$) always
refers to a solution to \eqref{eq:appRecurrences101} (resp., \eqref{eq:appRecurrences111}).
We first describe a minimality preserving transformation to obtain recurrences of a suitable form.


Given a solution \(\seq[\infty]{u_n}{n=-1}\) to \eqref{eq: ogform}, we define the sequence
\(\seq[\infty]{w_n}{n=-1}\) so that \(w_{-1}=v_{-1}\) and \(v_n = w_n \prod_{j=-1}^n \tfrac{g_2(j)}{2g_3(j)}\) for each \(n\in\N_0\).  It is easily shown that \(\seq{v_n}{n}\) satisfies recurrence \eqref{eq: ogform} if and only if \(\seq{w_n}{n}\) satisfies the following recurrence
\begin{equation*} \label{eq: kform}
			w_n = 2w_{n-1} + \frac{4g_1(n)g_3(n-1)}{g_2(n-1)g_2(n)} w_{n-2}
\end{equation*}

%
Let $\seq{w_n}{n}$
(resp., $\seq{w'_n}{n}$) be the sequence obtained by applying the above transformation to
$\seq{v_n}{n}$ (resp., $\seq{v'_n}{n}$). The recurrence relations satisfied by
$\seq{w_n}{n}$ and $\seq{w'_n}{n}$ take the respective forms
\begin{subequations}
\begin{align}
\label{eq:appKForm101}
	w_n &= 2w_{n-1} + \frac{4(n+\gamma) (n + \alpha - 1)}{\beta^2} w_{n-2}\\
\label{eq:appKForm111}
	w'_n &= 2w'_{n-1} - \frac{4(n+\gamma) (n + \alpha - 1)}{(2n + \beta)(2n + \beta - 2)} w'_{n-2}.
\end{align}
\end{subequations}

Now Kooman's characterisation deals with recurrences of the above form. 
In order to establish the asymptotic behaviour of solutions to recurrences \eqref{eq:appRecurrences101} and \eqref{eq:appRecurrences111}, it suffices to combine the
asymptotic equalities of solutions to \eqref{eq:appKForm101} and \eqref{eq:appKForm111} with the asymptotic behaviour of the product
$\prod_{j=-1}^n \tfrac{g_2(j)}{2g_3(j)}$ as $n \to \infty$. Let us first take care of the
asymptotics of the latter term. 

\begin{lemma}\hfil
\begin{enumerate}
\item We have $\prod_{j=-1}^n \frac{\beta}{2(n + \alpha)}\sim C\frac{\beta^n}{2^n n!} n^{-\alpha}$
for some constant $C \neq 0$.
\item We have $\prod_{j=-1}^n \frac{2n + \beta}{2(n + \alpha)}\sim C n^{\beta / 2 - \alpha}$ for some constant $C \neq 0$.
\end{enumerate}
\end{lemma}
\begin{proof}
The claims follow quite straightforwardly from the following observations.
First, for $a \neq 0$, $a n + b \neq 0$ for all $n \in\{-1,0,\ldots\}$ and
\begin{equation*}
\prod_{j=-1}^n (a j + b) = a^{n+2} \prod_{j=0}^{n+1}(j-1 + b/a)
= a^{n+2} (b/a - 1)_{n+1}
= a^{n+2} \frac{\Gamma( b/a + n)}{\Gamma(b/a - 1)}.
\end{equation*}
Second, by Stirling's formula, we have
\(\Gamma(x) \sim \sqrt{2\pi} \eu^{-x}x^{x - {1}/{2}}\) as  \(\Re(x)\to\infty\).
\end{proof}

We then establish the asymptotic behaviour of solutions to recurrences \eqref{eq:appKForm101} and \eqref{eq:appKForm111}.
\begin{lemma}\hfil
\begin{enumerate}
\item Recurrence relation \eqref{eq:appKForm101} admits two linearly independent solutions
that have the following asymptotic equalities
\begin{equation*}
w_n \sim (n-1)! ( \pm 2/\beta)^n
n^{\frac{1}{2}(\pm\beta + \gamma +  \alpha)  + 1 }.
\end{equation*}

%
%
\item Recurrence relation \eqref{eq:appKForm111} admits two linearly independent solutions
that have the following asymptotic equalities
\begin{equation*}
w'_n \sim n^{{1}/{4} + {(\alpha - \beta + \gamma)}/{2}} \exp(\pm 2 \sqrt{(\beta - \alpha - \gamma)n}).
\end{equation*}
\end{enumerate}
\end{lemma}
\begin{proof}
To apply Kooman's characterisation, we require knowledge of the asymptotic behaviour of the
coefficient of $w_{n-2}$ (resp., $w'_{n-2}$) in the corresponding recurrence relation.
In fact, Kooman studies recurrences of the form $x_{n+2} = 2 x_{n+1} - \mathfrak{C}_n x_n$.
(Notice the signature of this recurrence relation.) So, writing
$\mathfrak{C}_{n-2} = -\frac{4 g_3(n-1) g_1(n)}{g_2(n)g_2(n-1)}$,
we need the knowledge of the terms in the asymptotic expansion of $\mathfrak{C}_{n}$.

\begin{enumerate}
\item We may express $\mathfrak{C}_n = -\frac{4 (n+\alpha+1)(n+\gamma + 2)}{\beta^2} = -(\frac{2}{\beta})^2 n^2 - (\frac{2}{\beta})^2(\gamma + 3 +\alpha ) n + \mathcal{O}(1)$.
Now \cite[Ex.~1]{kooman2007asymptotic}, case $a=2$, establishes asymptotic equalities for solutions to this recurrence. (The parameters $C$ and $A$ there are assigned the values
$(\frac{2}{\beta})^2$ and $(\frac{2}{\beta})^2(\gamma + 3 +\alpha )$, respectively.)
Recalling that $\beta > 0$, the claimed asymptotic equalities are seen to hold after cancellations.

\item For $n$ large enough, we may express $\mathfrak{C}_n$ as a Laurent series:
$\mathfrak{C}_n = \frac{4(n + \alpha + 1)(n + \gamma + 2)}{(2n + \beta + 4)(2n + \beta + 2)} =  1 +  (\alpha - \beta + \gamma) \frac{1}{n} + \mathcal{O}(n^{-2})$. Now
\cite[Ex.~1]{kooman2007asymptotic}, case $a=-1$, establishes asymptotics for this recurrence,
as we assume $\beta - \alpha - \gamma \neq 0$.
(Again, the parameter $C$ there is assigned the value $ \beta - \alpha - \gamma$). The claimed
asymptotics equalities follow.\qedhere
\end{enumerate}
\end{proof}

The asymptotic equalities in \autoref{ex:KoomanAsymptotics} follow from the above two lemmas.
%

\bibliographystyle{plainurl}
\bibliography{biblio}

\end{document}